\documentclass{article}

\usepackage{latexsym}
\usepackage{amsmath, amsfonts, amsthm}
\usepackage{epsfig}
\usepackage{stmaryrd}
\usepackage{multicol}
\usepackage{pdfsync}
\usepackage{dsfont}
\usepackage{comment}
\usepackage{algorithmic}
\usepackage{algorithm}
\usepackage{caption}
\usepackage{psfrag}
\usepackage{subcaption}
\usepackage{xcolor}
\usepackage{graphics}

\newcommand{\norm}[1]{\left\lVert #1 \right\rVert}

\newtheorem{theorem}{Theorem}[section]
\newtheorem{corollary}[theorem]{Corollary}
\newtheorem{lemma}[theorem]{Lemma}
\newtheorem{proposition}[theorem]{Proposition}
\theoremstyle{definition}

\theoremstyle{remark}

\newtheorem{condition}[theorem]{Condition}
\numberwithin{equation}{section}

\title{Stochastic Gradient Descent in Continuous Time: A Central Limit Theorem}
\author{Justin Sirignano\footnote{Department of Industrial \& Systems Engineering, University of Illinois at Urbana-Champaign, Urbana, E-mail: jasirign@illinois.edu} \phantom{.}  and Konstantinos Spiliopoulos\footnote{Department of Mathematics and Statistics, Boston University, Boston, E-mail: kspiliop@math.bu.edu} \thanks{Research of K.S. supported in part by the National Science Foundation (DMS 1550918)} \thanks{The authors thank seminar participants at Princeton University and the University of Colorado Boulder for their comments.}  
}

\date{\today}

\begin{document}


\maketitle
\begin{abstract}
Stochastic gradient descent in continuous time (SGDCT) provides a computationally efficient method for the statistical learning of continuous-time models, which are widely used in science, engineering, and finance. The SGDCT algorithm follows a (noisy) descent direction along a continuous stream of data. The parameter updates occur in continuous time and satisfy a stochastic differential equation. This paper analyzes the asymptotic convergence rate of the SGDCT algorithm by proving a central limit theorem (CLT) for strongly convex objective functions and, under slightly stronger conditions, for non-convex objective functions as well. An L$^p$ convergence rate is also proven for the algorithm in the strongly convex case. The mathematical analysis lies at the intersection of stochastic analysis and statistical learning. 
\end{abstract}

\section{Introduction}

``Stochastic gradient descent in continuous time" (SGDCT) is a statistical learning algorithm for continuous-time models, which are common in science, engineering, and finance.  
Given a continuous stream of data, stochastic gradient descent in continuous time (SGDCT) can estimate unknown parameters or functions in stochastic differential equation (SDE) models. \cite{SGDCT1} analyzes the numerical performance of SGDCT for a number of applications in finance and engineering. We prove a central limit theorem and $L^p$ convergence rate for the SGDCT algorithm; see Section \ref{ContributionsIntro} for an overview of our results. Several technical challenges arise in the analysis, and our approaches may be more broadly useful for studying other continuous-time statistical learning methods (e.g., \cite{Raginsky}, \cite{Surace}, and \cite{Doya}).

Batch optimization for the statistical estimation of continuous-time models may be impractical for large datasets where observations occur over a long period of time.  Batch optimization takes a sequence of descent steps for the model error for the entire observed data path.  Since each descent step is for the model error for the \emph{entire observed data path}, batch optimization is slow (sometimes impractically slow) for long periods of time or models which are computationally costly to evaluate (e.g., partial differential equations or large systems of differential equations).

SGDCT provides a computationally efficient method for statistical learning over long time periods and for complex models.  SGDCT \emph{continuously} follows a (noisy) descent direction \emph{along the path of the observation}; this results in much more rapid convergence.  Parameters are updated online in continuous time, with the parameter updates $\theta_t$ satisfying a stochastic differential equation.

Consider a diffusion $X_t \in \mathcal{X}=\mathbb{R}^{m}$:
\begin{eqnarray}
d X_t = f^{\ast}(X_t) dt + \sigma d W_t.
\label{ClassofEqns}
\end{eqnarray}
The function $f^{\ast}(x)$ is unknown and $\sigma$ is a constant matrix. The goal is to statistically estimate a model $f(x, \theta)$ for $f^{\ast}(x)$ from the continuous stream of data $(X_t)_{t \geq 0}$. $W_t \in \mathbb{R}^m$ is a standard Brownian motion and we assume $\sigma$ is known.  The diffusion term $W_t$ represents any random behavior of the system or environment.  The functions $f(x, \theta)$ and $f^{\ast}(x)$ may be non-convex.

The stochastic gradient descent update in continuous time for the parameters $\theta \in \mathbb{R}^k$ satisfies the SDE:
\begin{eqnarray}
d \theta_t =  \alpha_{t}\big [  \nabla_{\theta}  f(X_t, \theta_t)(\sigma \sigma^{\top})^{-1} d X_t -    \nabla_{\theta}    f(X_t, \theta_t)(\sigma \sigma^{\top})^{-1} f(X_t, \theta_t)  dt\big],
\label{SDEMain}
\end{eqnarray}
where $\nabla_{\theta}  f(X_t; \theta_t)$ is matrix-valued and $\alpha_{t}$ is the learning rate. For example, $\alpha_t$ could equal $\frac{C_{\alpha}}{C_0 + t}$. \textcolor{black}{We assume that $\theta_0$ is initialized according to some distribution with compact support.}  The parameter update (\ref{SDEMain}) can be used for both statistical estimation given previously observed data as well as online learning (i.e., statistical estimation in real-time as data becomes available).  

Define the function
\[
g(x,\theta)=\frac{1}{2}\left\|f(x,\theta)-f^{\ast}(x)\right\|^{2}_{\sigma\sigma^{\top}}=\frac{1}{2}\left<f(x,\theta)-f^{\ast}(x),\left(\sigma\sigma^{\top}\right)^{-1}(f(x,\theta)-f^{\ast}(x))\right>,
\]
which measures the distance between the model $f(x, \theta)$ and the true dynamics $f^{\ast}(x)$ for a specific $x$.

We assume that $X_{t}$ is sufficiently ergodic (to be concretely specified later in the paper) and that it has some well-behaved $\pi(dx)$ as its unique invariant measure. As a general notation, if $h(x,\theta)$ is a generic $L^{1}(\pi)$ function, then we define its average over $\pi(dx)$ to be
\[
\bar{h}(\theta)=\int_{\mathcal{X}} h(x,\theta)\pi(dx).
\]
In particular, $\bar g ( \theta) = \int_{\mathcal{X}} g(x,\theta)\pi(dx)$ is a natural objective function to consider for our analysis of the asymptotic behavior of the algorithm $\theta_t$. $\bar g (\theta)$ is the weighted average of the distance between $f(x, \theta)$ and $f^{\ast}(x)$. The weights are given by $\pi(dx)$, which is the distribution that $X_t$ tends to as $t$ becomes large.

The distance $g(x, \theta)$ is decreased by moving $\theta$ in the descent direction $-\nabla_{\theta} g(x, \theta)$, which motivates the algorithm
\begin{eqnarray}
d \theta_t &=& - \alpha_t \nabla_{\theta} g(X_t, \theta_t) dt = \alpha_t  \nabla_{\theta} f(X_t, \theta_t ) (\sigma \sigma^{\top})^{-1}   \big{(}  f^{\ast}(X_t) - f(X_t, \theta_t )  \big{)} dt.
\label{Heuristic1}
\end{eqnarray}

$f^{\ast}(x)$ is unknown so (\ref{Heuristic1}) cannot be implemented in practice. However, $d X_t = f^{\ast}(X_t) dt + \sigma d W_t$ is a noisy estimate of $f^{\ast}(X_t) dt$, which can be used to derive the SGDCT algorithm (\ref{SDEMain}). In particular, it is easy to see that the SGDCT algorithm (\ref{SDEMain}) is the descent direction (\ref{Heuristic1}) plus a noise term:
\begin{eqnarray}
d \theta_t &=&  \alpha_{t}\big [  \nabla_{\theta}  f(X_t, \theta_t)(\sigma \sigma^{\top})^{-1} d X_t -    \nabla_{\theta}    f(X_t, \theta_t)(\sigma \sigma^{\top})^{-1} f(X_t, \theta_t)  dt\big] \notag \\
&=& \alpha_t \nabla_{\theta} f(X_t, \theta_t ) (\sigma \sigma^{\top})^{-1}   \big{(}  f^{\ast}(X_t) - f(X_t, \theta_t )  \big{)} dt +  \alpha_{t}  \nabla_{\theta}  f(X_t, \theta_t)(\sigma \sigma^{\top})^{-1} \sigma d W_t \notag \\
&=& - \alpha_t \nabla_{\theta} g (X_t, \theta_t) dt +  \alpha_{t}  \nabla_{\theta}  f(X_t, \theta_t)(\sigma \sigma^{\top})^{-1} \sigma d W_t.
\label{DeComposition1}
\end{eqnarray}

The descent direction $- \alpha_t \nabla_{\theta} f (X_t, \theta_t)$ in equation (\ref{DeComposition1}) depends upon $X_t$, so it is unclear in the above formulation if $\theta_t$ makes progress towards a fixed point. In order to understand the behavior of $\theta_t$, it is helpful to decompose (\ref{DeComposition1}) into several terms:
\begin{eqnarray}
d \theta_t = - \underbrace{ \alpha_t \nabla_{\theta} \bar g (\theta_t) dt}_{\textrm{Descent term}} -  \underbrace{ \alpha_t \big{(} \nabla_{\theta} g(X_t, \theta_t)   - \nabla_{\theta} \bar g ( \theta_t) \big{)}  dt }_{ \textrm{Fluctuation term} } + \underbrace{ \alpha_{t}  \nabla_{\theta}  f(X_t, \theta_t)(\sigma \sigma^{\top})^{-1} \sigma d W_t}_{\textrm{Noise term} }.
\label{DeComposition2}
\end{eqnarray}

Heuristically, if $\alpha_t$ decays with time (e.g., $\alpha_{t}=\frac{C_{\alpha}}{C_0 + t}$), the descent term $-\alpha_t \nabla_{\theta} \bar g (\theta_t)$ will dominate the fluctuation and noise terms for large $t$. Then, one might expect that $\theta_t$ will converge to a local minimum of $\bar g (\theta)$. The authors proved in \cite{SGDCT1} that $\theta_t$ converges to a critical point of the objective function $\bar g (\theta)$:
\begin{eqnarray}
 \norm{ \nabla \bar g(\theta_t) } \overset{a.s.} \rightarrow 0 \phantom{....} \textrm{as} \phantom{....} t \rightarrow \infty.
 \end{eqnarray}

However, \cite{SGDCT1} left as an open question whether $\theta_t$ satisfies any asymptotic convergence rate. In this paper, we prove a central limit theorem and an $L^p$ convergence rate, which we present in the next subsection of this Introduction.



 \subsection{Contributions of this paper} \label{ContributionsIntro}
 We prove a central limit theorem for $\theta_t$ when $\bar g(\theta)$ has a single critical point $\theta^{\ast}$:
 \begin{eqnarray}
 \sqrt{t} \big{(} \theta_t - \theta^{\ast} \big{)} \overset{d} \rightarrow \mathcal{N}(0, \bar \Sigma) \phantom{....} \textrm{as} \phantom{....} t \rightarrow \infty.
 \end{eqnarray}

This result is proven for objective functions $\bar g (\theta)$ which may be non-convex and models $f(x; \theta)$ with up to linear growth in $\theta$ and polynomial growth in $x$; see Theorem \ref{T:MainTheorem3}. Furthermore, when $\bar g ( \theta)$ is strongly convex we prove an $L^p$ convergence rate:
 \begin{eqnarray}
 \mathbb{E} \big{[} \norm{ \theta_t - \theta^{\ast} }^p \big{]} \leq \frac{K}{\left(C_{0}+t\right)^{p/2}},
 \end{eqnarray}
 for $p\geq 1$.  We prove this result for models $f(x, \theta)$ with up to quadratic growth in $\theta$ and polynomial growth in $x$. In addition, in this strongly convex case, we prove the central limit theorem also holds for models $f(x, \theta)$ with up to quadratic growth in $\theta$ and polynomial growth in $x$. The $L^p$ convergence rate and CLT for the strongly convex case are stated in Theorems \ref{T:MainTheorem1} and \ref{T:MainTheorem2}, respectively.

Theorems \ref{T:MainTheorem1} and \ref{T:MainTheorem2} do not make use of the results in \cite{SGDCT1}. As a part of the proof for Theorem \ref{T:MainTheorem3}, we also strengthen the convergence result of \cite{SGDCT1}, which did not allow $f(x, \theta)$ to grow in $\theta$; see Theorem \ref{T:AlternateMainTheoremSGD1}. Theorems  \ref{T:MainTheorem1},  \ref{T:MainTheorem2}, and \ref{T:MainTheorem3} are proven for the learning rate $\alpha_t = \frac{C_{\alpha}}{C_0 + t}$. Analogous results for Theorems \ref{T:MainTheorem1}, \ref{T:MainTheorem2}, and \ref{T:MainTheorem3} hold of course for a general class of learning rates $\alpha_{t}$; see Proposition \ref{C:GeneralLearningRate}. The precise statement of the mathematical results and the technical assumptions required are presented in Section \ref{MainResults}.

In addition, as a corollary, our results prove an $L^p$ convergence rate and a CLT for the case where there is no dependence on $X_t$. That is, if $\theta^{\ast}$ is the unique critical point of the function $g(\theta)$ and
\begin{eqnarray}
d \theta_t = \alpha_t  \big{(} -\nabla g(\theta_t) dt + d W_t \big{)},
\end{eqnarray}
then $\theta_t \overset{a.s.} \rightarrow \theta^{\ast}$ and $\sqrt{t} \big{(} \theta_t - \theta^{\ast} \big{)} \overset{d} \rightarrow \mathcal{N}(0, \bar \Sigma)$ as $t \rightarrow \infty$. In addition, if $g(\theta)$ is strongly convex, then $\mathbb{E} \big{[} \norm{ \theta_t - \theta^{\ast} }^p \big{]} \leq \frac{K}{\left(C_{0}+t\right)^{p/2}}$.

These mathematical results are important for two reasons.  First, they establish theoretical guarantees for the rate of convergence of the algorithm.  Secondly, they can be used to analyze the effects of different features such as the learning rate $\alpha_t$, the level of noise $\sigma$, and the shape of the objective function $\bar g(\theta)$.  We are able to precisely characterize the regime under which the optimal convergence rate exists as well as characterize the limiting covariance $\bar \Sigma$.  The regime depends entirely upon the choice of the learning rate.



%


Proving the central limit theorem is challenging due to the nature of the $X_t$ process. The data $X_t$ will be correlated over time, which differs from the standard discrete-time version of stochastic gradient descent where the data is usually considered to be  i.i.d. at every step. In particular, the fluctuation term $\alpha_t \big{(} \nabla_{\theta} g(X_t, \theta_t)   - \nabla_{\theta} \bar g ( \theta_t) \big{)}  dt$ must be analyzed and shown to become small in some appropriate sense. We evaluate, and control with rate $\alpha_t^2$, these fluctuations using a Poisson partial differential equation.

Although it can be heuristically seen that the descent term dominates the fluctuation and noise terms in equation (\ref{DeComposition2}) for large $t$, the descent term $\alpha_t \nabla_{\theta} \bar g (\theta_t) \rightarrow 0$ as $t \rightarrow \infty$ since $\displaystyle \lim_{t \rightarrow \infty} \alpha_t = 0$. Therefore, it is a priori unclear if asymptotically $\alpha_t \nabla_{\theta} \bar g (\theta_t)$ will remain sufficiently large to guarantee an asymptotic convergence rate for $\theta_t$.

\cite{SGDCT1} proved the convergence result $\norm{ \nabla_{\theta} \bar g( \theta_t) } \overset{a.s.} \rightarrow 0$ using a cycle of stopping times to capture when $\norm{ \nabla_{\theta} \bar g( \theta_t) }$ is large or small. However, this approach is not useful for proving the CLT. Instead, we derive a stochastic integral to represent $\sqrt{t} ( \theta_t - \theta^{\ast} )$ using Duhamel's principle and the fundamental solution of the random ODE $d \Psi_t = -  \alpha_t \nabla_{\theta} \bar g ( \tilde \theta_t)  \Psi_t dt$ where $\tilde \theta_t$ lies on a line connecting $\theta^{\ast}$ and $\theta_t$. The integrand of this stochastic integral includes the fluctuation and noise terms as well as $\Psi_t$. 

The proofs in this paper also require addressing several other challenges. The model $f(x, \theta)$ is allowed to grow with $\theta$.  This means that the fluctuations as well as other terms can grow with $\theta$.  Therefore, we must prove an a priori stability estimate for $\norm{ \theta_t } $.  Proving a central limit theorem for the non-convex $\bar g(\theta)$ in Theorem \ref{T:MainTheorem3} is not straightforward since the convergence speed of $\theta_t$ can become arbitrarily slow in certain regions, and the gradient can even point away from the global minimum $\theta^{\ast}$. To address this, we consider the stochastic integral after the time $\tau_{\delta}$, which is defined as \emph{the final time} $\theta_t$ enters a neighborhood of $\theta^{\ast}$. However,  $\tau_{\delta}$
is not a stopping time, and therefore careful analysis is required to study the limiting behavior of the stochastic integral.

The approaches we develop in this paper are likely to be of general use for studying the asymptotics of other continuous-time statistical learning methods (e.g., \cite{Raginsky}, \cite{Surace}, and \cite{Doya}).


\subsection{Literature Review} \label{LiteratureReview}
The vast majority of the statistical learning, machine learning, and stochastic gradient descent literature addresses discrete-time algorithms.  In contrast, this paper analyzes a statistical learning algorithm in continuous time. We review the existing literature that is most relevant to our work.  We also comment on the importance of developing and analyzing continuous-time algorithms for addressing continuous-time models.

Many discrete-time papers study algorithms for $\theta_n$ without the $X$-dynamics (for example, stochastic gradient descent with i.i.d. noise at each step).  The inclusion of the $X$-dynamics makes analysis significantly more challenging.  An  L$^2$ convergence rate and central limit theorem result for discrete-time stochastic gradient descent is presented in \cite{Benveniste}.  Our setup and assumptions are different from \cite{Benveniste}. Our proof approach is different than the one in \cite{Benveniste} and leverages the continuous-time nature of our setting, which is the formulation of interest in many engineering and financial problems (see \cite{SGDCT1}).

\cite{YinGupta} studied convergence of a continuous-time version of the stochastic gradient descent algorithm. They proved convergence for the process and a central limit theorem for the time-averaged process. There are a couple of differences between our paper and \cite{YinGupta}. Our paper includes the $X$-dynamics, which are not considered in \cite{YinGupta}. \cite{YinGupta} proves a central limit theorem for the time-averaged process, while we study the central limit theorem for the process itself. In addition, our assumptions are different.

\cite{Sakrison} studied convergence of a continuous-time version of the stochastic gradient descent algorithm. However, \cite{Sakrison} did not include the $X$-dynamics. \cite{Sakrison} only proves convergence while our paper proves a central limit theorem and a convergence rate.

%

\cite{Raginsky} studies continuous-time stochastic mirror descent in a setting different than ours.  In the framework of \cite{Raginsky}, the objective function is known.  In this paper, we consider the statistical estimation of the unknown dynamics of a random process (i.e. the $X$ process satisfying (\ref{ClassofEqns})).

Statisticians and financial engineers have actively studied parameter estimation of SDEs, although typically not with statistical learning or machine learning approaches.  The likelihood function will usually be calculated from the \emph{entire observed path of $X$} (i.e., batch optimization) and then maximized to find the maximum likelihood estimator (MLE).  For example, \cite{Basawa} derives the likelihood function for the entire path of a continuously observed $X$. Unlike in this paper, the actual optimization procedure to maximize the likelihood function is often not analyzed. Readers are referred to \cite{bishwal2008parameter,Kutoyants,rao1999statistical} for thorough reviews of classical statistical inference methods for stochastic differential equations.


Continuous-time models are common in engineering and finance.  There are often coefficients or functions in these models which are uncertain or unknown; stochastic gradient descent can be used to learn these model parameters from data.

It is natural to ask  why use SGDCT versus a straightforward approach which (1) discretizes the continuous-time dynamics and then (2) applies traditional stochastic gradient descent.  We elaborated in detail in regards to this issue in \cite{SGDCT1}, where specific examples are provided there to showcase the differences. For completeness, let us briefly discuss the issues that arise.

SGDCT allows for the application of numerical schemes of choice to the theoretically correct statistical learning equation for continuous-time models.   This can lead to more accurate and more computationally efficient parameter updates.  Numerical schemes are always applied to continuous-time dynamics and different numerical schemes may have different properties for different continuous-time models.  A priori performing a discretization to the system dynamics and then applying a traditional discrete-time stochastic gradient descent scheme can result in a loss of accuracy, or may not even converge, see \cite{SGDCT1}.  For example, there is no guarantee that (1) using a higher-order accurate scheme to discretize the system dynamics and then (2) applying traditional stochastic gradient descent will produce a statistical learning scheme which is higher-order accurate in time.  Hence, it makes sense  to first develop the continuous-time statistical learning equation, and then apply the higher-order accurate numerical scheme.

In addition to model estimation, SGDCT can be used to solve continuous-time optimization problems, such as American options. In \cite{SGDCT1}, SGDCT was combined with a deep neural network to solve American options in up to $100$ dimensions.  Alternatively, one could  discretize the dynamics and then use the Q-learning algorithm (traditional stochastic gradient descent applied to an approximation of the discrete HJB equation).  However, as we showed in \cite{SGDCT1}, Q-learning is biased while SGDCT is unbiased.  Furthermore, in SDE models with Brownian motions, the Q-learning algorithm can blow up as the time step size $\Delta$ becomes small; see \cite{SGDCT1} for details.


\subsection{Organization of Paper}
In Section \ref{MainResults} we state our assumptions and the main results of this paper. The proof of the L$^p$ convergence rate for $p\geq 1$ is in Section \ref{L2proof}.  The proof of the CLT in the strongly convex case is in Section \ref{CLTproof}.  Section \ref{QuasiConvexproof} proves the central limit theorem for a class of non-convex models.  In Section \ref{ConvergenceAnalysis}, the convergence rate results are used to analyze the behavior and dynamics of the stochastic gradient descent in continuous-time algorithm.  Some technical results required for the proofs are presented in Appendix \ref{Preliminary}. Appendix \ref{ConvergenceProofQuadraticGrowth} contains the proof of Theorem \ref{T:AlternateMainTheoremSGD1}, which strengthens the convergence result of \cite{SGDCT1}.  In particular, Appendix \ref{ConvergenceProofQuadraticGrowth} provides the necessary adjustments to the proofs of \cite{SGDCT1} in order to guarantee convergence in the case where the model $f(x,\theta)$ is allowed to grow with respect to $\theta$.

\section{Main Results} \label{MainResults}
We prove three main results.  Theorem \ref{T:MainTheorem1} is an L$^p$ convergence for the strongly convex case.  Theorem \ref{T:MainTheorem2} is a central limit theorem for the strongly convex case.  Theorem \ref{T:MainTheorem3} is a central limit theorem for the non-convex case with a single critical point.

We say that ``the function $h(\theta)$ is strongly convex with constant $C$" if there exists a $C > 0$ such that $z^{\top} \Delta h(\theta) z \geq   C z^{\top} z$ for any non-zero $z \in \mathbb{R}^k$.  $\Delta h$ is the Hessian matrix of the function $h(\theta)$. Conditions \ref{A:Assumption1} and \ref{A:Assumption2} require that $C C_{\alpha} > 1$ where $C_{\alpha}$ is the magnitude of the learning rate and $C$ is the strong convexity constant for the objective function $\bar{g}(\theta^{*})$ at point $\theta=\theta^{*}$.  This is an important conclusion of the convergence analysis in this paper: the learning rate needs to be sufficiently large in order to achieve the optimal rate of convergence. This dependence of the convergence rate on the learning rate is not specific to the algorithm SGDCT in this paper, but applies to other algorithms (even deterministic gradient descent algorithms).  We discuss this in more detail in Section \ref{ConvergenceAnalysis}.


It should be emphasized that the assumptions for  Theorem \ref{T:MainTheorem1} and Theorem \ref{T:MainTheorem2} (the strongly convex case) allow for the model $f(x, \theta)$ to grow up to quadratically in $\theta$. On the other hand, Theorem \ref{T:MainTheorem3} (the non-convex case) is proven under the assumption that $f(x, \theta)$   grows up to linearly in $\theta$. The growth of the model $f(x, \theta)$ is allowed to be polynomial in $x$ for Theorems  \ref{T:MainTheorem1}, \ref{T:MainTheorem2}, and  \ref{T:MainTheorem3}.

The proofs of Theorems \ref{T:MainTheorem1}, \ref{T:MainTheorem2}, and \ref{T:MainTheorem3} are in Sections \ref{L2proof}, \ref{CLTproof}, and \ref{QuasiConvexproof}, respectively. Some technical results required for these proofs are presented in Appendix \ref{Preliminary}. Appendix \ref{ConvergenceProofQuadraticGrowth} contains the changes necessary to generalize the proof of convergence of \cite{SGDCT1} from the case where $\bar g (\theta)$ is bounded to the case where $\bar{g}(\theta)$ can grow up to quadratically in $\theta$; see Theorem \ref{T:AlternateMainTheoremSGD1} for the corresponding rigorous statement.

Let us now list our conditions.  Condition \ref{A:LyapunovCondition} guarantees uniqueness and existence of an invariant measure for the $X$ process.   Condition \ref{A:LyapunovCondition} and the second part of Condition \ref{A:Assumption0} guarantee that equation (\ref{ClassofEqns}) is well-posed. We will use the notation $\norm{ \cdot }$ for the Euclidean norms of both vectors and matrices. $\norm{\cdot}$ will denote the $\ell^2$-norm and $\norm{\cdot}_p$ will denote the $\ell^p$-norm.

\begin{condition}\label{A:LyapunovCondition}
We assume that $\sigma\sigma^{\top}$ is non-degenerate constant diffusion matrix and $\lim_{|x|\rightarrow\infty}f^{\ast}(x)\cdot x=-\infty$.
\end{condition}

In regards to regularity of the involved functions, we impose the following Condition \ref{A:Assumption0}.
\begin{condition}
\label{A:Assumption0}
\begin{enumerate}
\item  We assume that $\nabla_{\theta}g(x,\cdot)\in \mathcal{C}^{2}(\mathbb{R}^{k})$ for all $x\in\mathcal{X} \textcolor{black}{\subseteq \mathbb{R}^m}$, $\frac{\partial^{2}\nabla_{\theta}g}{\partial x^{2}}\in \mathcal{C}\left(\mathcal{X},\mathbb{R}^{k}\right)$, $\nabla_{\theta}g(\cdot,\theta)\in \mathcal{C}^{\alpha}\left(\mathcal{X}\right)$ uniformly in $\theta\in\mathbb{R}^{k}$ for some $\alpha\in(0,1)$.
\item The function $f^{\ast}(x)$ is  $\mathcal{C}^{2+\alpha}(\mathcal{X})$ with $\alpha\in(0,1)$. Namely, it has two derivatives in $x$, with all partial derivatives being H\"{o}lder continuous, with exponent $\alpha$, with respect to $x$.
\item For every $t\geq 0$, equation (\ref{SDEMain})  has a unique strong solution.
\end{enumerate}
\end{condition}

Conditions \ref{A:RecurrenceCondition0} and \ref{A:GrowthConditions0} are imposed in order to guarantee that $\theta_{t}$ has bounded moments uniformly in time $t$.
\begin{condition}\label{A:RecurrenceCondition0}
There exists a constant $R<\infty$ and an almost everywhere positive function $\kappa(x)$ such that
\begin{align}
\left<-\nabla_{\theta}g(x,\theta),\theta/\norm{\theta}\right>&\leq - \kappa(x) \norm{\theta}, \textrm{ for }\norm{\theta}\geq R.\nonumber
\end{align}
\end{condition}

\begin{condition}\label{A:GrowthConditions0}
Consider the function
\begin{align}
\tau(x,\theta)&=\left<\nabla_{\theta}f(x,\theta)\nabla_{\theta}f^{\top}(x,\theta)\frac{\theta}{\norm{\theta}},\frac{\theta}{\norm{\theta}}\right>^{1/2}.
\end{align}
Then, there exists a function $\lambda(x)$, growing not faster than polynomially in $\norm{x}$, such that for any $\theta_{1},\theta_{2}\in\mathbb{R}^{k}$ and $x\in\mathbb{R}^{m}$
\begin{align}
|\tau(x,\theta_{1})-\tau(x,\theta_{2})|&\leq |\lambda(x)|\rho\left(\norm{\theta_{1}-\theta_{2}}\right)
\end{align}
where $\rho(\xi)$ is an increasing function on $[0,\infty)$ with $\rho(0)=0$ and $\int_{\xi>0}\rho^{-2}(\xi)d\xi=\infty$.
\end{condition}


\textcolor{black}{Notice that Condition \ref{A:GrowthConditions0} simplifies considerably when $\theta$ is one-dimensional. Condition \ref{A:GrowthConditions0} is always true in the case where $\nabla_{\theta}f(x,\theta)$ is independent of $\theta$ (which happens for example when $f(x,\theta)$ is affine in $\theta$). Furthermore, if $f(x, \theta)$ and $\nabla_{\theta} f(x, \theta)$ are uniformly bounded in $\theta$ and polynomially bounded in $x$, Conditions \ref{A:RecurrenceCondition0} and \ref{A:GrowthConditions0} are not required for the CLT in the non-convex case (Theorem \ref{T:MainTheorem3}).}

The Conditions \ref{A:RecurrenceCondition0} and \ref{A:GrowthConditions0} are necessary to prove that  $ \sup_{t \geq 0} \mathbb{E}[ \norm{ \theta_t }^p ] < K$ for $p \geq 2$.  This uniform bound on moments, which we prove in Appendix \ref{StabilityBound}, is in turn required for the proofs of the L$^p$ convergence and central limit theorem in Theorems \ref{T:MainTheorem1} and \ref{T:MainTheorem2}.  The function $g(x, \theta)$ is allowed to grow, which means that some a priori bounds on the moments of $\theta_t$ are necessary.


In regards to the learning rate we assume
 \begin{condition}
\label{A:AssumptionLR}
The learning rate is $\alpha_t = \frac{ C_{\alpha}}{C_0+ t}$ where $C_{\alpha} > 0$ and $C_0 > 0$ are constants.
\end{condition}
Condition \ref{A:AssumptionLR} makes the presentation of the results easier. However, as we shall see in Corollary \ref{C:GeneralLearningRate} the specific form $\alpha_t = \frac{ C_{\alpha}}{C_0+ t}$ is not necessary as long as $\alpha_{t}$ satisfies certain conditions. We chose to present the results under this specific form both for presentation purposes and because this is the usual form that the learning rate takes in practice. However, we do present the result for a general learning rate in Corollary \ref{C:GeneralLearningRate}.


\begin{condition}
\label{A:Assumption1}
\begin{enumerate}
\item $\norm{\nabla_{\theta}^{i}f(x,\theta)}\leq K \left(1+\norm{x}^{q}+\norm{\theta}^{(2-i)\vee 0}\right)$, for $i=0,1,2,3$  and for some finite constants $K,q<\infty$.
\item  $\bar g (\theta)$ is strongly convex with constant $C$.
\item $C C_{\alpha}> 1$.
\item $\bar g (\theta) \in C^3$ and $\norm{\nabla_{\theta}^{i}\bar{g}(\theta)}\leq K \left(1+\norm{\theta}^{4-i}\right)$, for $i=0,1,2,3$  and for some finite constant $K<\infty$.
\end{enumerate}
\end{condition}

  As will be seen from the proof of the Theorems \ref{T:MainTheorem1} and \ref{T:MainTheorem2}, one needs to control terms of the form $\int_{0}^{t} \textcolor{black}{\alpha_{s}}(\nabla \bar{g}(\theta_{s})-g(X_{s},\theta_{s}))ds$. Due to ergodicity of the $X$ process one expects that such terms are small in magnitude and go to zero as $t\rightarrow\infty$. However, the speed at which they go to zero is what matters here. We treat such terms by rewriting them equivalently using appropriate Poisson type partial differential equations (PDE). Conditions \ref{A:LyapunovCondition}, \ref{A:Assumption0} and \ref{A:Assumption1}  guarantee that these Poisson equations have unique solutions that do not grow faster than polynomially in the $x$ and $\theta$ variables (see Theorem \ref{T:RegularityPoisson} in Appendix \ref{Preliminary}).

\begin{theorem}\label{T:MainTheorem1}
Assume that Conditions \ref{A:LyapunovCondition}, \ref{A:Assumption0}, \ref{A:RecurrenceCondition0}, \ref{A:GrowthConditions0},  \ref{A:AssumptionLR} and \ref{A:Assumption1} hold. Then, for $p\geq 1$, there exists a constant $0<K<\infty$ such that
\begin{eqnarray*}
\mathbb{E}[ \norm{ \theta_t - \theta^{\ast} }^{p} ] \leq \frac{K}{\left(C_0 + t\right)^{p/2}}.
\end{eqnarray*}
\end{theorem}

In order to state  the central limit theorem results we need to introduce some notation. Let us denote by $v(x,\theta)$ the solution to the Poisson equation (\ref{Eq:CellProblem}) with $H(x, \theta) = \nabla_{\theta} g (x, \theta) - \nabla_{\theta} \bar g ( \theta)$. Let us also set
\[
h(x, \theta) =  \bigg{(}  \nabla_{\theta} f(x,\theta) (\sigma \sigma^{\top} )^{-1}   -  \nabla_x v(x,\theta) \bigg{)} \sigma\sigma^{\top}\bigg{(}  \nabla_{\theta} f(x,\theta)  (\sigma \sigma^{\top} )^{-1}   -  \nabla_x v(x,\theta) \bigg{)}^{\top}\text{ and }\bar h(\theta) = \int h(x, \theta) \pi(dx).
\]

\begin{theorem}\label{T:MainTheorem2}
Assume that Conditions \ref{A:LyapunovCondition}, \ref{A:Assumption0}, \ref{A:RecurrenceCondition0}, \ref{A:GrowthConditions0},  \ref{A:AssumptionLR} and \ref{A:Assumption1} hold. 
Then,
\begin{eqnarray*}
\sqrt{t} ( \theta_t - \theta^{\ast} ) \overset{d} \rightarrow \mathcal{N}(0,  \bar \Sigma ),
\end{eqnarray*}
where $\bar\Sigma$ is defined to be
\[
\bar{\Sigma}=C_{\alpha}^{2}\int_{0}^{\infty}e^{-s\left(C_{\alpha}\Delta\bar{g}(\theta^{\ast})-I\right)}\bar{h}(\theta^{\ast})e^{-s\left(C_{\alpha}\left(\Delta\bar{g}\right)^{\top}(\theta^{\ast})-I\right)}ds.
\]
\end{theorem}


Our results of course immediately imply an $L^p$ convergence rate and a CLT for the case when there is no dependence on $X_t$.
\begin{corollary}
Assume that $\theta^{\ast}$ is the unique critical point of the function $g(\theta)$ and that
\begin{eqnarray}
d \theta_t = \alpha_t \big{(} -\nabla g(\theta_t) dt +  d W_t \big{)}.
\end{eqnarray}
has a unique strong solution. In addition, we assume that Conditions  \ref{A:Assumption0}.1, \ref{A:RecurrenceCondition0} (with $g(\theta)$ in place of $g(x,\theta)$),   \ref{A:AssumptionLR} and \ref{A:Assumption1}.2-\ref{A:Assumption1}.4 (with $g(\theta)$ in place of $\bar{g}(\theta)$) hold. Then, $\theta_t$ satisfies the $L^p$ convergence rate
\begin{eqnarray*}
\mathbb{E}[ \norm{ \theta_t - \theta^{\ast} }^{p} ] \leq \frac{K}{\left(C_0 + t\right)^{p/2}}.
\end{eqnarray*}
and the CLT
\begin{eqnarray*}
\sqrt{t} ( \theta_t - \theta^{\ast} ) \overset{d} \rightarrow \mathcal{N}(0,  \bar \Sigma ),
\end{eqnarray*}
where here
\[
\bar{\Sigma}=C_{\alpha}^{2}\int_{0}^{\infty}e^{-s\left(C_{\alpha}\Delta g(\theta^{\ast})-I\right)}e^{-s\left(C_{\alpha}\left(\Delta g\right)^{\top}(\theta^{\ast})-I\right)}ds.
\]
\end{corollary}

The main result of \cite{SGDCT1} states that if $\bar{g}(\theta)$ and its derivatives are bounded with respect to $\theta$, then under assumptions of ergodicity for the $X$ process one has convergence of the algorithm, in the sense that \textcolor{black}{$\lim_{t \rightarrow \infty} \norm{\nabla_{\theta} \bar g (\theta_t) } = 0$}. In this paper, we allow growth of $\bar{g}(\theta)$ with respect to $\norm{\theta}$. In particular, as we shall state in Theorem \ref{T:AlternateMainTheoremSGD1} and prove in Appendix \ref{ConvergenceProofQuadraticGrowth}, the results of \cite{SGDCT1} hold true without considerable extra work if one allows up to linear growth for $f(x,\theta)$ with respect to $\theta$, which translates into up to quadratic growth for $\bar{g}(\theta)$ and up to linear growth of $\nabla \bar{g}(\theta)$ with respect to $\theta$. Let us formalize the required assumptions in the form of Condition \ref{A:AssumptionSGD2} below, which also strengthens Condition \ref{A:Assumption1}.
\begin{condition}
\label{A:AssumptionSGD2}
\begin{enumerate}
\item $\norm{\nabla_{\theta}^{i}f(x,\theta)}\leq K \left(1+\norm{x}^{q}+\norm{\theta}^{(1-i)\vee 0}\right)$, for $i=0,1,2$  and for some finite constants $K,q<\infty$.
\item $\norm{\nabla_{\theta}^{i}\bar{g}(\theta)}\leq K \left(1+\norm{\theta}^{2-i}\right)$, for $i=0,1,2$  and for some finite constant $K<\infty$.
\item $\nabla\bar g (\theta)$ is globally Lipschitz.
\end{enumerate}
\end{condition}

%

\begin{theorem}\label{T:AlternateMainTheoremSGD1}
Assume that Conditions \ref{A:LyapunovCondition}, \ref{A:Assumption0}, \ref{A:RecurrenceCondition0},  \ref{A:GrowthConditions0}, \ref{A:AssumptionLR} and \ref{A:AssumptionSGD2} hold. Then,
\begin{eqnarray*}
\lim_{t \rightarrow \infty} \textcolor{black}{ \norm{\nabla_{\theta} \bar g (\theta_t) } }= 0, \phantom{....} a.s.
\end{eqnarray*}
\end{theorem}
Theorem \ref{T:AlternateMainTheoremSGD1} proves convergence for non-convex $\bar g(\theta)$ even when $f(x, \theta)$ grows at most linearly in $\theta$.  Theorem \ref{T:AlternateMainTheoremSGD1} is proven in Appendix \ref{ConvergenceProofQuadraticGrowth}.  The proof is based on the uniform bound for the moments of $\theta$ established in Appendix \ref{Preliminary}.


Theorem \ref{T:AlternateMainTheoremSGD1} is required for proving a central limit theorem for non-convex $\bar g (\theta)$.  The central limit theorem for non-convex $\bar g (\theta)$ is proven in Theorem \ref{T:MainTheorem3}.
\begin{condition}
\label{A:Assumption2}
\begin{enumerate}
\item $\bar g(\theta)$ may be non-convex but has a single critical point $\theta^{\ast}$.
\item There is a $\delta^{*}>0$ small enough such that  $\bar g (\theta)$ is strongly convex in the region $\norm{ \theta - \theta^{\ast} } < \delta^{\ast}$ with constant $C$ and $CC_{\alpha}>1$.
\item $\bar g (\theta) \in C^3$.
\item $\nabla_{\theta_i} \bar g(\theta) > 0$ if $\theta_i - \theta^{\ast}_i > R_0$ and $\nabla_{\theta_i} \bar g(\theta) < 0$ if $\theta_i - \theta^{\ast}_i < -R_0$ for $i =1, 2, \ldots, k$ and for some $R_0>0$ large enough.
\end{enumerate}
\end{condition}
The fourth part in Condition \ref{A:Assumption2} guarantees that $\nabla_{\theta_i} \bar g (\theta)$ always points inwards towards the global minimum $\theta^{\ast}$ if $| \theta_i|$ is sufficiently large.

\begin{theorem}\label{T:MainTheorem3}
Assume that Conditions \ref{A:LyapunovCondition}, \ref{A:Assumption0},  \ref{A:RecurrenceCondition0},  \ref{A:GrowthConditions0}, \ref{A:AssumptionLR}, \ref{A:AssumptionSGD2} and \ref{A:Assumption2} hold.   Then, we have that
\begin{eqnarray*}
\sqrt{t} ( \theta_t - \theta^{\ast} ) \overset{d} \rightarrow \mathcal{N}(0, \bar \Sigma ),
\end{eqnarray*}
where $\bar\Sigma$ is defined as in Theorem \ref{T:MainTheorem2}.
\end{theorem}
If $f(x, \theta)$ and $\nabla_{\theta} f(x, \theta)$ are uniformly bounded in $\theta$ and polynomially bounded in $x$, Theorem \ref{T:MainTheorem3} is true without Conditions \ref{A:RecurrenceCondition0} and \ref{A:GrowthConditions0}.

Proposition \ref{C:GeneralLearningRate} shows that under certain conditions on the learning rate $\alpha_{t}$, the specific form of the learning rate assumed in Condition \ref{A:AssumptionLR} is not necessary.  In particular, the convergence rate and central limit theorem results can be proven for a general learning rate $\alpha_t$. The proof of Proposition \ref{C:GeneralLearningRate} follows exactly the same steps as the proofs of Theorems \ref{T:MainTheorem1}, \ref{T:MainTheorem2}, and \ref{T:MainTheorem3} albeit more tedious algebra and is omitted.
\begin{proposition}\label{C:GeneralLearningRate}
 Let us denote $\Psi^{(p)}_{t,s}=e^{-pC\int_{s}^{t}\alpha_{\rho}d\rho}, \text{ for } p\geq 1$ and consider the matrix-valued solution $\Phi^{*}_{t,s}$  to equation (\ref{FundamentalSolution2}). Theorems \ref{T:MainTheorem1}, \ref{T:MainTheorem2}, and \ref{T:MainTheorem3} also hold under a general learning rate $\alpha_t$ satisfying the conditions:
\begin{eqnarray}
&& \int_{0}^{\infty}\alpha_{t}dt=\infty, \int_{0}^{\infty}\alpha^{2}_{t}dt<\infty, \int_{0}^{\infty}|\alpha_{s}^{\prime}|ds<\infty, \exists p>0 \text{ such that }\lim_{t\rightarrow\infty}\alpha_{t}t^{p}=0,\notag\\
&& \forall p\geq 2\quad \int_{0}^{t} \left(\alpha^{2}_{s}+|\alpha_{s}^{\prime}|\right) \Psi^{(p)}_{t,s} \alpha_{s}^{p/2-1} ds  \leq O(\alpha^{p/2}_{t}), \text{ and } \int_{0}^{t} \alpha^{2}_{s} \Psi^{(1)}_{t,s}  ds  \leq o(\alpha^{1/2}_{t}) \notag\\ 
&& \int_{0}^{t} \alpha^{5/2}_{s} \Psi^{(2)}_{t,s}  ds  \leq o(\alpha_{t}),\qquad \int_{0}^{t} \alpha_{s}^{2} \norm{\Phi^{*}_{t,s}}^2 ds = O(\alpha_{t}), \notag \\
&& \forall p\geq 2\quad \Psi^{(p)}_{t,0}\leq O(\alpha^{p/2}_{t}) \text{ and } \Psi^{(1)}_{t,0}  \leq o(\alpha_{t}^{1/2}).\nonumber
\end{eqnarray}
  In particular, the statements of Theorems \ref{T:MainTheorem1}, \ref{T:MainTheorem2}, and \ref{T:MainTheorem3} then take the form
\begin{eqnarray*}
\mathbb{E}[ \norm{ \theta_t - \theta^{\ast} }^{p} ] \leq K\alpha_{t}^{p/2}
\end{eqnarray*}
and
\begin{eqnarray*}
\alpha^{-1/2}_{t} ( \theta_t - \theta^{\ast} ) \overset{d} \rightarrow \mathcal{N}(0, \bar \Sigma ),
\end{eqnarray*}
where now $\bar{\Sigma}=\left[\bar{\Sigma}_{i,j}\right]_{i,j=1}^{k}$ is as in (\ref{Eq:LimitingSigma}) but with the bracket term $\left[ \frac{C_{\alpha}^{2}}{\left(\lambda_{m}+\lambda_{m'}\right)C_{\alpha}-1}\right] $ replaced by \[\left[\lim_{t\rightarrow\infty}\alpha_{t}^{-1}\int_{0}^{t}\alpha_{s}^{2}e^{-(\lambda_{m}+\lambda_{m'})\int_{s}^{t}\alpha_{u}du}ds\right].\]
\end{proposition}

%
%
It is easy to check that if we use $\alpha_{s}=\frac{C_{\alpha}}{C_{0}+s}$ as the learning rate, then the conditions that appear in Proposition \ref{C:GeneralLearningRate} all hold if $CC_{\alpha}>1$.

We conclude this section by mentioning that in the bounds that appear in the subsequent sections, $0<K<\infty$ will denote unimportant fixed constants (that do not depend on $t$ or other important parameters). The constant $K$ may change from line to line but it will always be denoted by the same symbol $K$.  Without loss of generality and in order to simplify notation in the proofs, we will let $C_{0}=0$, consider $t\geq 1$ (i.e., the initial time is set to $t = 1$), and let $\sigma$ be the identity matrix.

\section{Theorem \ref{T:MainTheorem1} - L$^p$ convergence rate in strongly convex case} \label{L2proof}
The proofs in this paper will repeatedly make use of two important uniform moment bounds. First, as we  prove  in Appendix \ref{StabilityBound}, we have, for $p \geq 2$, that
\begin{eqnarray*}
\displaystyle \sup_{t \geq 0} \mathbb{E}[ \norm{ \theta_t }^p ] < K,
\end{eqnarray*}
Second, it is  known from \cite{PardouxVeretennikov1} that under the imposed conditions for the $X$ process, for $p \geq 2$,
\begin{eqnarray*}
\displaystyle \sup_{t \geq 0} \mathbb{E}[ \norm{ X_t }^p ] < K.
\end{eqnarray*}

To begin the proof for the L$^p$ convergence rate, re-write the algorithm (\ref{SDEMain}) for $\theta_t$ in terms of  $\bar g(\theta)$.
\begin{eqnarray}
d \theta_t &=& \alpha_t \nabla_{\theta} f(X_t, \theta_t)  ( dX_t - f(X_t, \theta_t)  dt) \notag \\
& =& \alpha_t  \nabla_{\theta} f(X_t, \theta_t)  ( f^{\ast}(x) - f(X_t, \theta_t) )dt +  \alpha_t  \nabla_{\theta} f(X_t, \theta_t) d W_t \notag \\
&=&- \alpha_t  \nabla_{\theta} g(X_t, \theta_t) dt +  \alpha_t  \nabla_{\theta} f(X_t, \theta_t) d W_t  \notag \\
&=&- \alpha_t  \nabla_{\theta} \textcolor{black}{\bar g (\theta_t) }  +   \alpha_t \big{(} \nabla_{\theta} \bar g (\theta_t ) - \nabla_{\theta} g(X_t, \theta_t) \big{)} dt +  \alpha_t  \nabla_{\theta} f(X_t, \theta_t) d W_t.
\label{AlgorithmRewrite}
\end{eqnarray}

A Taylor expansion yields:
\begin{eqnarray*}
\nabla_{\theta} \bar g (\theta_t) &=&  \nabla_{\theta} \bar g (\theta^{\ast} ) + \Delta \bar g (\theta_t^1) ( \theta_t - \theta^{\ast} ) = \Delta \bar g (\theta_t^1) ( \theta_t - \theta^{\ast} ),
\end{eqnarray*}
where $\theta_t^1$ is an appropriately chosen point in the segment connecting $ \theta_t$ and $ \theta^{\ast}$. Substituting this Taylor expansion into equation (\ref{AlgorithmRewrite}) produces the equation:
\begin{eqnarray*}
d  ( \theta_t  - \theta^{\ast} )&=& - \alpha_t \Delta \bar g( \theta_t^1) ( \theta_t  - \theta^{\ast})    dt   + \alpha_t ( \nabla_{\theta} \bar g(\theta_t) - \nabla_{\theta} g (X_t, \theta_t ) ) dt + \alpha_t \nabla_{\theta} f(X_t, \theta_t) d W_t.
\end{eqnarray*}

Let $Y_t = \theta_t - \theta^{\ast}$.  Then, $Y_t$ satisfies the SDE
\begin{eqnarray*}
d Y_t = - \alpha_t \Delta \bar g( \theta_t^1) Y_t   dt   + \alpha_t ( \nabla_{\theta} \bar g(\theta_t) - \nabla_{\theta} g (X_t, \theta_t ) ) dt  + \alpha_t \nabla_{\theta} f(X_t, \theta_t) d W_t.
\end{eqnarray*}

By It\^{o} formula, we then have for $p\geq 2$
\begin{align}
d\norm{Y_{t}}^{p}&=p\norm{Y_{t}}^{p-2}\sum_{i,k}Y_{t}^{k}\alpha_{t}\left(\nabla_{\theta}f(X_{t},\theta_{t})\right)_{k,i}dW_{t}^{i}-p\alpha_{t}\norm{Y_{t}}^{p-2}\left<Y_{t},\Delta \bar{g}(\theta^{1}_{t})Y_{t}\right>dt\nonumber\\
&+p\alpha_{t}\norm{Y_{t}}^{p-2}\left<Y_{t}, \nabla_{\theta} \bar g(\theta_t) - \nabla_{\theta} g (X_t, \theta_t )\right>dt\nonumber\\
&+\frac{p}{2}\alpha_{t}^{2}\norm{Y_{t}}^{p-2}
\left[\sum_{i,k}\left(\left(\nabla_{\theta}f(X_{t},\theta_{t})\right)_{i,k}\right)^{2}+(p-2)\sum_{i}\left(\sum_{k}\norm{Y_{t}}^{-1}Y_{t}^{k}\left(\nabla_{\theta}f(X_{t},\theta_{t})\right)_{k,i}\right)^{2}\right]dt.\label{Eq:Ynorm}
\end{align}

Using the strong convexity of $\bar{g}$ we obtain the inequality
\begin{align}
d\norm{Y_{t}}^{p}&\leq p\norm{Y_{t}}^{p-2}\sum_{i,k}Y_{t}^{k}\alpha_{t}\left(\nabla_{\theta}f(X_{t},\theta_{t})\right)_{k,i}dW_{t}^{i}-pC \alpha_{t}\norm{Y_{t}}^{p}
dt\nonumber\\
&+p\alpha_{t}\norm{Y_{t}}^{p-2}\left<Y_{t}, \nabla_{\theta} \bar g(\theta_t) - \nabla_{\theta} g (X_t, \theta_t )\right>dt\nonumber\\
&+\frac{p}{2}\alpha_{t}^{2}\norm{Y_{t}}^{p-2}
\left[\sum_{i,k}\left(\left(\nabla_{\theta}f(X_{t},\theta_{t})\right)_{i,k}\right)^{2}+(p-2)\sum_{i}\left(\sum_{k}\norm{Y_{t}}^{-1}Y_{t}^{k}\left(\nabla_{\theta}f(X_{t},\theta_{t})\right)_{k,i}\right)^{2}\right]dt.\nonumber
\end{align}

Let's now define the process
\begin{align}
M_{t}&=p\int_{1}^{t}e^{-pC\int_{s}^{t}\alpha_{\rho}d\rho}\alpha_{s}\norm{Y_{s}}^{p-2}\sum_{i,k}Y_{s}^{k}\left(\nabla_{\theta}f(X_{s},\theta_{s})\right)_{k,i}dW_{s}^{i},\nonumber
\end{align}
and notice that $M_{t}$ solves the SDE
\begin{align}
dM_{t}&=-pC\alpha_{t}M_{t}dt+p\norm{Y_{t}}^{p-2}\sum_{i,k}Y_{t}^{k}\left(\nabla_{\theta}f(X_{t},\theta_{t})\right)_{k,i}dW_{t}^{i}.\nonumber
\end{align}

Next, if we set $\Gamma_{t}=\norm{Y_{t}}^{p}-M_{t}$ we obtain that
\begin{align}
d\Gamma_{t}&\leq -pC \alpha_{t}\Gamma_{t}dt+p\alpha_{t}\norm{Y_{t}}^{p-2}\left<Y_{t}, \nabla_{\theta} \bar g(\theta_t) - \nabla_{\theta} g (X_t, \theta_t )\right>dt\nonumber\\
&+\frac{p}{2}\alpha_{t}^{2}\norm{Y_{t}}^{p-2}
\left[\sum_{i,k}\left(\left(\nabla_{\theta}f(X_{t},\theta_{t})\right)_{i,k}\right)^{2}+(p-2)\sum_{i}\left(\sum_{k}\norm{Y_{t}}^{-1}Y_{t}^{k}\left(\nabla_{\theta}f(X_{t},\theta_{t})\right)_{k,i}\right)^{2}\right]dt.\nonumber
\end{align}

Next, we define the function
\begin{align*}
\Psi^{(p)}_{t,s}=e^{-pC\int_{s}^{t}\alpha_{\rho}d\rho}=\left(\frac{s}{t}\right)^{pCC_{\alpha}},
\end{align*}
and the comparison principle gives
\begin{align}
\Gamma_{t}&\leq \Psi^{(p)}_{t,1}\norm{Y_{1}}^{p}+p\int_{1}^{t}\left[\alpha_{s}\Psi^{(p)}_{t,s}\norm{Y_{s}}^{p-2}\left<Y_{s}, \nabla_{\theta} \bar g(\theta_s) - \nabla_{\theta} g (X_s, \theta_s )\right>\right]ds\nonumber\\
&+\frac{p}{2}\int_{1}^{t}\Psi^{(p)}_{t,s}\alpha_{s}^{2}\norm{Y_{s}}^{p-2}
\left[\sum_{i,k}\left(\left(\nabla_{\theta}f(X_{s},\theta_{s})\right)_{i,k}\right)^{2}+(p-2)\sum_{i}\left(\sum_{k}\norm{Y_{s}}^{-1}Y_{s}^{k}\left(\nabla_{\theta}f(X_{s},\theta_{s})\right)_{k,i}\right)^{2}\right]ds\nonumber\\
&=\Gamma_{t}^{1}+\Gamma_{t}^{2}+\Gamma_{t}^{3}.\label{Eq:GammaInequality}
\end{align}

The next step is to rewrite the second term of (\ref{Eq:GammaInequality}), i.e., $\Gamma_{t}^{2}= \textcolor{black}{p} \int_{1}^{t}\left[\alpha_{s}\Psi^{(p)}_{t,s}\norm{Y_{s}}^{p-2}\left<Y_{s}, \nabla_{\theta} \bar g(\theta_s) - \nabla_{\theta} g (X_s, \theta_s )\right>\right]ds$.  We construct the corresponding Poisson equation and use its solution to analyze $\Gamma_{t}^{2}$. Define $G(x, \theta) =  \left<\theta-\theta^{*},\nabla_{\theta} \bar g(\theta) - \nabla_{\theta} g (x, \theta ) \right>$ and let $v(x, \theta)$ be the solution to the PDE $\mathcal{L}_x v(x, \theta) = G(x, \theta)$. Here $\mathcal{L}_x$ is the infinitesimal generator of the $X$ process. Due to Theorem \ref{T:RegularityPoisson}, the Poisson PDE solution satsfies:
\begin{eqnarray}
&&  \norm{v (x, \theta) } + \norm{ \nabla_x v (x, \theta) }  \leq  K ( 1+  \norm{ \theta}^{m_1} ) (1 + \norm{x}^{m_2}) ,  \notag \\
&& \sum_{i=1,2} \norm{ \frac{\partial^i v}{\partial \theta^i }(x, \theta) } + \norm{ \frac{\partial^2 v}{\partial x \partial \theta}(x, \theta) }  + \norm{ \frac{\partial^3 v}{\partial x \partial \theta^2}(x, \theta) }   \leq  K ( 1+  \norm{ \theta}^{m_3} )( 1 + \norm{x}^{m_4} )
\label{BoundL2PDE}
\end{eqnarray}
for appropriate, but unimportant for our purposes, constants $m_{1},m_{2},m_{3},m_{4}$. By It\^{o}'s formula:
\begin{eqnarray*}
v(X_t, \theta_t) - v(X_s, \theta_s) &=&  \int_s^t \mathcal{L}_{x} v(X_{u},\theta_{u})du+\int_s^t   \mathcal{L}_{\theta } v(X_{u},\theta_{u})du  \notag \\
&+& \int_{s}^t \nabla_{x} v(X_{u},\theta_{u}) dW_{u} +  \int_{s}^t  \alpha_u \nabla_{\theta} v (X_{u},\theta_{u})  \nabla_{\theta} f(X_u, \theta_u) dW_{u} \notag \\
& + & \int_s^t \alpha_u \nabla_{\theta x} v(X_u, \theta_u) \nabla_{\theta} f(X_u, \theta_u) du,
\end{eqnarray*}
\textcolor{black}{where $\mathcal{L}_{\theta}$ is the infinitesimal generator of the $\theta_t$ process and $\mathcal{L}_x$ is the infinitesimal generator of the $X_t$ process.}

Define $v_t \equiv v(X_t, \theta_t)$ and recognize that:
\begin{eqnarray}
G (X_t, \theta_t)  dt  &=& \mathcal{L}_{x} v(X_{t},\theta_{t}) dt \notag \\
&=& d v_t - \mathcal{L}_{\theta } v(X_{t},\theta_{t}) dt - \nabla_{x} v(X_{t},\theta_{t}) dW_{t} -  \alpha_t \nabla_{\theta} v (X_{t},\theta_{t})  \nabla_{\theta} f(X_t, \theta_t) dW_{t}  \notag \\
&-& \alpha_t \nabla_{\theta x} v(X_t, \theta_t) \nabla_{\theta} f(X_t, \theta_t) dt.
\label{PoissonErgodic1}
\end{eqnarray}

Using this result, $\Gamma_t^2$ can be rewritten as:
\begin{eqnarray}
\Gamma_{t}^{2} &=&  \int_{1}^{t}\left[\alpha_{s}\Psi^{(p)}_{t,s}\norm{Y_{s}}^{p-2}\left<Y_{s}, \nabla_{\theta} \bar g(\theta_s) - \nabla_{\theta} g (X_s, \theta_s )\right>\right]ds \notag \\
&=&  \int_1^t \alpha_s \Psi^{(p)}_{t,s}\norm{Y_{s}}^{p-2}  d v_s - \int_1^t \alpha_s  \Psi^{(p)}_{t,s}\norm{Y_{s}}^{p-2}  \nabla_x v(X_s, \theta_s) d W_s \notag \\
&-&  \int_1^t \alpha_s \Psi^{(p)}_{t,s}\norm{Y_{s}}^{p-2}  \mathcal{L}_{\theta} v (X_s, \theta_s) ds -  \int_1^t \alpha_s^2  \Psi^{(p)}_{t,s}\norm{Y_{s}}^{p-2}  \nabla_{\theta} v (X_{s},\theta_{s})  \nabla_{\theta} f(X_s, \theta_s) dW_{s} \notag\\
& -&  \int_1^t \alpha_s^2 \Psi^{(p)}_{t,s}\norm{Y_{s}}^{p-2} \nabla_{\theta x} v(X_s, \theta_s) \nabla_{\theta} f(X_s, \theta_s) ds \notag \\
&=& \Gamma_t^{2,1} + \Gamma_t^{2,2}+ \Gamma_t^{2,3} + \Gamma_t^{2,4} + \Gamma_t^{2,5}.
\label{Eq:Gamma2}
\end{eqnarray}

Let's first rewrite  the first term $ \Gamma_t^{2,1}$.  We apply It\^{o}'s formula to $\alpha_s \Psi^{(p)}_{t,s} \norm{Y_{s}}^{p-2} v_s$:
\begin{align*}
\alpha_t \Psi^{(p)}_{t,t}  \norm{Y_{t}}^{p-2} v_t  &- \alpha_1 \Psi^{(p)}_{t,1}  \norm{Y_{1}}^{p-2} v_1  = \int_1^t \alpha_s  \Psi^{(p)}_{t,s} \norm{Y_{s}}^{p-2} d v_s - \int_1^t \frac{C_{\alpha}}{s^2} \Psi^{(p)}_{t,s} \norm{Y_{s}}^{p-2}  v_s ds \nonumber\\
& +  \int_1^t \alpha_s  \frac{\partial \Psi^{(p)}_{t,s}  }{\partial s} \norm{Y_{s}}^{p-2} v_s ds+ \int_1^t \alpha_s  \Psi^{(p)}_{t,s} v_{s}d\norm{Y_{s}}^{p-2}+\int_1^t \alpha_s  \Psi^{(p)}_{t,s} d\left[\norm{Y_{s}}^{p-2} , v_s\right].
\end{align*}

Then, we have the following representation for $\Gamma_t^{2,1}$:
\begin{align}
 \Gamma_t^{2,1} &=  \int_1^t \alpha_s  \Psi^{(p)}_{t,s} \norm{Y_{s}}^{p-2} d v_s  \notag \\
 &=   \alpha_t \Psi^{(p)}_{t,t} \norm{Y_{t}}^{p-2}  v_t  - \alpha_1 \Psi^{(p)}_{t,1} \norm{Y_{1}}^{p-2}  v_1  +  \int_1^t \frac{C_{\alpha}}{s^2} \Psi^{(p)}_{t,s} \norm{Y_{s}}^{p-2}  v_s ds  - \int_1^t \alpha_s  \frac{\partial \Psi^{(p)}_{t,s}  }{\partial s} \norm{Y_{s}}^{p-2}v_s ds\nonumber\\
 &- \int_1^t \alpha_s  \Psi^{(p)}_{t,s} v_{s}d\norm{Y_{s}}^{p-2}-\int_1^t \alpha_s  \Psi^{(p)}_{t,s} d\left[\norm{Y_{s}}^{p-2} , v_s\right]\nonumber\\
 &=   \alpha_t \Psi^{(p)}_{t,t} \norm{Y_{t}}^{p-2}  v_t  - \alpha_1 \Psi^{(p)}_{t,1} \norm{Y_{1}}^{p-2}  v_1  +  C_{\alpha}^{-1}\int_1^t \alpha_{s}^{2} \Psi^{(p)}_{t,s} \norm{Y_{s}}^{p-2}  v_s ds  - pC\int_1^t \alpha^{2}_s   \Psi^{(p)}_{t,s}  \norm{Y_{s}}^{p-2}v_s ds\nonumber\\
 &- \int_1^t \alpha_s  \Psi^{(p)}_{t,s} v_{s}d\norm{Y_{s}}^{p-2}-\int_1^t \alpha_s  \Psi^{(p)}_{t,s} d\left[\norm{Y_{s}}^{p-2} , v_s\right].\label{Eq:Gamma21term}
 \end{align}

Now, we are ready to put things together. Equation (\ref{Eq:Ynorm}) with $p-2$ in place of $p$ is then used to evaluate the second to last term of (\ref{Eq:Gamma21term}) and similarly for the quadratic covariation term $d\left[\norm{Y_{s}}^{p-2} , v_s\right]$ of the last term of (\ref{Eq:Gamma21term}). Plugging (\ref{Eq:Gamma21term}) in (\ref{Eq:Gamma2}) and that in (\ref{Eq:GammaInequality}), we get that there is an unimportant  constant $K<\infty$ large enough and a matrix-valued function $\zeta(x,\theta)$ that grows at most polynomially in $x$ and $\theta$ such that
\begin{align}
\Gamma_{t}&\leq \Psi^{(p)}_{t,1}\norm{Y_{1}}^{p}+\alpha_t \Psi^{(p)}_{t,t} \norm{Y_{t}}^{p-2}  v_t  - \alpha_1 \Psi^{(p)}_{t,1} \norm{Y_{1}}^{p-2}  v_1\nonumber\\
&+K\int_{1}^{t}\alpha_{s}^{2}\Psi^{(p)}_{t,s}\norm{Y_{s}}^{p-2}\zeta(X_{s},\theta_{s})ds+ \hat{M}_{t},
\label{Eq:GammaInequality2}
\end{align}
where $\hat{M}_{t}$ is a mean zero and square integrable (this follows from the uniform moment bounds on the $X$ and $\theta$ processes) Brownian stochastic intergal.

Recall now that we want to evaluate $\mathbb{E}\norm{Y_{t}}^{p}$. Recalling the definition of $\Gamma_{t}$ and taking expectation in (\ref{Eq:GammaInequality2}) we obtain
\begin{align*}
\mathbb{E}\norm{Y_{t}}^{p}&\leq \mathbb{E}\left[\Psi^{(p)}_{t,1}\norm{Y_{1}}^{p}+\alpha_t \norm{Y_{t}}^{p-2}  v_t  - \alpha_1 \Psi^{(p)}_{t,1} \norm{Y_{1}}^{p-2}  v_1\right]+  K\mathbb{E} \int_{1}^{t}\alpha_{s}^{2}\Psi^{(p)}_{t,s}\norm{Y_{s}}^{p-2}\zeta(X_{s},\theta_{s})ds.
\end{align*}

Recalling that $\Psi^{(p)}_{t,1}=t^{-pCC_{\alpha}}$ and that $CC_{a}>1$ we get that $\Psi^{(p)}_{t,1}\leq t^{-p}$. Hence we have obtained that for any $p\geq 2$ the following inequality holds
\begin{align}
\mathbb{E}\norm{Y_{t}}^{p}&\leq K t^{-p}+\mathbb{E}\left[\alpha_t \norm{Y_{t}}^{p-2}  v_t  \right]+  K\mathbb{E} \int_{1}^{t}\alpha_{s}^{2}\Psi^{(p)}_{t,s}\norm{Y_{s}}^{p-2}\zeta(X_{s},\theta_{s})ds.\label{Eq: ExpectedValueBound2}
\end{align}

The next step is to proceed by induction. Using the uniform moment bounds for $X$ and $\theta$ together with the polynomial growth of $v(x,\theta)$ and $\zeta(x,\theta)$, we get for $p=2$
\begin{align}
\mathbb{E}\norm{Y_{t}}^{2}&\leq K t^{-2}+\mathbb{E}\left[\alpha_t  v_t  \right]+  K\mathbb{E} \int_{1}^{t}\alpha_{s}^{2}\Psi^{(2)}_{t,s}\zeta(X_{s},\theta_{s})ds\nonumber\\
&\leq K \left[t^{-2} +t^{-1}+ t^{-2CC_{\alpha}}\int_{1}^{t} s^{2CC_{\alpha}-2}ds\right]\leq K t^{-1},\nonumber
\end{align}
where the unimportant constant $K$ may change from line to line. Hence the desired statement is true for $p=2$. Next let us assume that it is true for exponent $p-1$ and we want to prove that it is true for exponent $p$. Using H\"{o}lder inequality with exponents $r_{1},r_{2}>1$ such that $1/r_{1}+1/r_{2}=1$ and choosing $r_{1}=\frac{p-1}{p-2}>1$,
\begin{align*}
\mathbb{E}\left[\alpha_t \norm{Y_{t}}^{p-2}  v_t\right]&\leq \alpha_{t} \left(\mathbb{E} \norm{Y_{t}}^{p-1}\right)^{1/r_{1}}  \left(\mathbb{E} |v_{t}|^{r{_2}}\right)^{1/r_{2}}\nonumber\\
&\leq Kt^{-1}t^{-\frac{p-2}{2}}= K t^{-p/2},
\end{align*}
and likewise
\begin{align*}
\mathbb{E}\norm{\int_{1}^{t}\alpha_{s}^{2}\Psi^{(p)}_{t,s}\norm{Y_{s}}^{p-2}\zeta(X_{s},\theta_{s})ds}&\leq \int_{1}^{t}\alpha_{s}^{2}\Psi^{(p)}_{t,s}\left(\mathbb{E}\norm{Y_{s}}^{(p-2)r_{1}}\right)^{1/r_{1}}\left(\mathbb{E}\norm{\zeta(X_{s},\theta_{s})}^{r_{2}}\right)^{1/r_{2}}ds\nonumber\\
&\leq K\int_{1}^{t} \alpha_{s}^{2}\Psi^{(p)}_{t,s}\left(\mathbb{E}\norm{Y_{s}}^{(p-2)r_{1}}\right)^{1/r_{1}}ds\nonumber\\
&\leq K\int_{1}^{t} \alpha_{s}^{2}\Psi^{(p)}_{t,s}\frac{1}{s^{(p-2)/2}}ds\nonumber\\
&\leq Kt^{-pCC_{\alpha}}\int_{1}^{t} s^{pCC_{\alpha}-1-\frac{p}{2}}ds\nonumber\\
&\leq K\left(t^{-p/2}-t^{-pCC_{\alpha}}\right)\nonumber\\
&\leq Kt^{-p/2}.
\end{align*}

Putting the last two displays together, (\ref{Eq: ExpectedValueBound2}) gives
\begin{align*}
\mathbb{E}\norm{Y_{t}}^{p}&\leq K t^{-p/2},
\end{align*}
which is the statement of Theorem \ref{T:MainTheorem1} for integer $p\geq 2$. The statement for any $p\geq 1$ then follows from H\"{o}lder inequality.
 This concludes the proof of the theorem.

\section{Proof of Central Limit Theorem in strongly convex case} \label{CLTproof}
To prove the central limit theorem, we use a \emph{second-order} Taylor expansion:
\begin{eqnarray*}
\nabla_{\theta} \bar g (\theta_t) &=&  \nabla_{\theta} \bar g (\theta^{\ast} ) + \Delta \bar g (\theta^{\ast}) ( \theta_t - \theta^{\ast} ) + \frac{1}{2} \frac{\partial^3 \bar g}{\partial \theta^3}(\theta^1_t) ( \theta_t - \theta^{\ast} )  ( \theta_t - \theta^{\ast} )^{\top} ,
\end{eqnarray*}
where $\theta_t^1$  is an appropriate point chosen in the segment that connects $\theta_t$ to $\theta^{\ast}$. \textcolor{black}{The last term on the RHS of the above equation is a tensor-matrix product.} Then, the evolution of $\theta_t$ follows:
\begin{eqnarray*}
d  ( \theta_t  - \theta^{\ast} )&=& - \alpha_t \Delta \bar g( \theta^{\ast}) ( \theta_t  - \theta^{\ast})    dt - \frac{\alpha_t}{2} \frac{\partial^3 \bar g}{\partial \theta^3}  (\theta^1_t) ( \theta_t - \theta^{\ast} )  ( \theta_t - \theta^{\ast} )^{\top}  dt   + \alpha_t ( \nabla_{\theta} \bar g(\theta_t) - \nabla_{\theta} g (X_t, \theta_t ) ) dt \notag \\
& &+ \alpha_t \nabla_{\theta} f(X_t, \theta_t) d W_t.
\end{eqnarray*}

Let $Y_t = \theta_t - \theta^{\ast}$.  Then, $Y_t$ satisfies the SDE
\begin{eqnarray*}
d Y_t = - \alpha_t \Delta \bar g( \theta_t^\ast) Y_t   dt  - \frac{\alpha_t}{2} \frac{\partial^3 \bar g }{\partial \theta^3} (\theta^1_t) Y_t Y_t^{\top}   dt   + \alpha_t ( \nabla_{\theta} \bar g(\theta_t) - \nabla_{\theta} g (X_t, \theta_t ) ) dt  + \alpha_t \nabla_{\theta} f(X_t, \theta_t) d W_t.
\end{eqnarray*}

Let $\Phi^{*}_{t,s} \in \mathbb{R}^{k \times k}$ be the fundamental solution satisfying
\begin{eqnarray}
d \Phi^{*}_{t,s} &=& - \alpha_t \Delta \bar g( \theta^{\ast})  \Phi^{*}_{t,s} dt, \qquad
\Phi^{*}_{s,s} = I,
\label{FundamentalSolution2}
\end{eqnarray}
where $I$ is the identity matrix. As in Section \ref{L2proof} we set without loss of generality $C_{0}=0$ and assume that the initial time is at $t=1$. Then, $Y_t$ can be written in terms of $\Phi^{*}_{t,s}$:
\begin{eqnarray}
Y_t &=&  \Phi^{*}_{t,1} Y_1    - \int_1^t \Phi^{*}_{t,s}  \frac{\alpha_s}{2} \frac{\partial^3 \bar g}{\partial \theta^3} (\theta^1_s) Y_s Y_s^{\top} ds  +   \int_1^t \Phi^{*}_{t,s} \alpha_s ( \nabla_{\theta} \bar g(\theta_s) - \nabla_{\theta} g (X_s, \theta_s ) ) ds + \int_1^t   \Phi^{*}_{t,s}   \alpha_s \nabla_{\theta} f(X_s, \theta_s) d W_s \notag \\
&=& \Gamma^1_t + \Gamma_t^2 + \Gamma_t^3 + \Gamma_t^4.
\label{DuhamelTwo}
\end{eqnarray}

Next, a convergence rate for $\Phi^{*}_{t,s}$ must be established in the matrix norm $\norm{A} = \sqrt{ \sum_{i,j} A_{i,j}^2 }$.   Consider the change of variables $\tau(t) = s + \int_s^t \alpha_u du$ and define $\tilde \Phi^{*}_{\tau, s}$ via $\tilde \Phi^{*}_{\tau, s} = \Phi^{*}_{t, s}$.  For $\alpha_t = \frac{C_{\alpha}}{t}$, we have $\tau(t)=s+ \textcolor{black}{ C_{\alpha} ( \log(t) - \log(s) ) }$.  Performing this change of variables,
\begin{eqnarray}
d \tilde \Phi^{*}_{\tau,s} &=& - \Delta \bar g(\theta^{\ast}) \tilde \Phi^{*}_{\tau,s} d \tau,  \qquad
\tilde \Phi^{*}_{s,s} = I,\label{Eq:TransformedPhiMatrix}
\end{eqnarray}
for $\tau \geq s$. 
Define $\tilde \Phi^{*,j}_{\tau,s}$ as the $j$-th column of $\tilde \Phi^{*}_{\tau,s}$ and $\Delta \bar g( \theta^{\ast} )^i$ as the $i$-th row of the matrix $ \Delta \bar g( \theta^{\ast} ) $.  Consider the one-dimensional differential equation:
\begin{eqnarray*}
\frac{d }{ d \tau} \norm{ \tilde \Phi^{*,j}_{\tau,s} }^2 &=& \frac{d }{ d \tau} \bigg{(}  \big{(} \tilde \Phi^{*,j,1}_{\tau,s}  \big{)}^2 + \cdots + \big{(} \tilde \Phi^{*,j,k}_{\tau,s} \big{)}^2 \bigg{)}  \notag \\
&=& -2  \bigg{(} \tilde \Phi^{*,j,1}_{\tau,s}   \Delta \bar g( \theta^{\ast} )^1 \tilde \Phi^{*,j}_{\tau,s}  + \cdots +  \tilde \Phi^{*,j,k}_{\tau,s}  \Delta \bar g( \theta^{\ast} )^k \tilde \Phi^{*,j}_{\tau,s} \bigg{)} \notag \\
&=& -2  \bigg{(} \tilde \Phi^{*,j}_{\tau,s } \bigg{)} ^{\top}  \Delta \bar g( \theta^{\ast} ) \tilde \Phi^{*,j}_{\tau,s} \textcolor{black}{\leq} -2 C \bigg{(} \tilde \Phi^{*,j}_{\tau,s} \bigg{)}^{\top}    \tilde \Phi^{*,j}_{\tau,s} = -2 C \norm{ \tilde \Phi^{*,j}_{\tau,s} }^2,
\end{eqnarray*}
where we have used the strong convexity assumption.  This of course yields the convergence rate:
\begin{eqnarray*}
\norm{ \tilde \Phi^{*}_{\tau,s} }^2  \leq   K  e^{-2 C (\tau-s) }.
\end{eqnarray*}

Changing variables again yields the convergence rate in the original time coordinate $t$:
\begin{eqnarray}
\norm{  \Phi^{*}_{t,s} }^2  &=& \norm{ \tilde \Phi^{*}_{\tau,s} }^2  \leq  K e ^{-2 C ( \tau-s ) }  \leq K e^{-2 C   \int_s^t \textcolor{black}{ \alpha_u} du }    =K t^{-2 C  C_{\alpha} } s^{ 2 C C_{\alpha } }.
\label{PhiBoundTwo}
\end{eqnarray}

We recall some important properties of $\Phi_{t,s}^{\ast}$ (for reference, see Proposition 2.14 in \cite{Chicone}).  $\Phi_{t,s}^{\ast}$ is  differentiable in $t$ with the semi-group property $\Phi_{t,1}^{\ast} = \Phi_{t,s}^{\ast} \Phi_{s,1}^{\ast}$.  Furthermore, $\Phi_{t,s}^{\ast}$ is invertible with inverse $\Phi_{t,s}^{\ast, -1} = \Phi_{s,t}^{\ast}$.  Therefore,
\begin{eqnarray}
\frac{\partial}{\partial s} \Phi_{t,s}^{\ast} &=&  \frac{\partial}{\partial s} \bigg{(}  \Phi_{t,1}^{\ast}  \Phi_{s,1}^{\ast, -1}   \bigg{)}\notag\\
 &= &\Phi_{t,1}^{\ast} \frac{\partial}{\partial s}  \Phi_{s,1}^{\ast, -1} = \Phi_{t,1}^{\ast}   \Phi_{s,1}^{\ast, -1}   \frac{\partial \Phi_{s,1}^{\ast} }{\partial s}    \Phi_{s,1}^{\ast, -1}   \notag \\
&=& - \Phi_{t,s}^{\ast}  \bigg{(}  \alpha_s \Delta \bar g( \theta^{\ast})  \Phi_{s,1}^{\ast}  \bigg{)}    \Phi_{s,1}^{\ast, -1}\notag\\
& =&- \alpha_s  \Phi_{t,s}^{\ast}  \Delta \bar g( \theta^{\ast} ).
\label{MatrixCalculation1TimeDeriv}
\end{eqnarray}

Combining the calculation (\ref{MatrixCalculation1TimeDeriv}) with the bound (\ref{PhiBoundTwo}) yields:
\begin{eqnarray*}
\norm{\frac{\partial}{\partial s} \Phi_{t,s}^{\ast} }^2 \leq K \textcolor{black}{\alpha_s^2} t^{-2 C  C_{\alpha} } s^{ 2 C C_{\alpha } }.
\end{eqnarray*}


Now, let's examine the process $\sqrt{t} Y_t $, which represents the fluctuations of $\theta_t$ around the global minimum $\theta^{\ast}$.  First, observe that $\sqrt{t} \Gamma^1_t \overset{a.s.} \rightarrow 0$ due to the bound (\ref{PhiBoundTwo}).  Let's now examine $\sqrt{t} \Gamma_t^2$.  Recall that $\mathbb{E}[ \norm{Y_t }^p] \leq \frac{K}{t^{p/2}}$ from the L$^p$ convergence rate in the previous section. Applying this result with $p=2,3$ we obtain
\begin{eqnarray}
\mathbb{E} \bigg{[} \norm{ \sqrt{t} \Gamma_t^2  }_1 \bigg{]} &\leq& \mathbb{E} \bigg{[} \sqrt{t} \int_1^t \norm{ \Phi^{*}_{t,s} \alpha_s \frac{\partial^3 \bar g}{\partial \theta^3}  (\theta^1_s) Y_s Y_s^{\top} }_1 ds \bigg{]} \notag \\
&\leq& K \sqrt{t}  \int_1^t \norm{ \Phi^{*}_{t,s}  } \left(\frac{1}{s^2}+\frac{1}{s^{5/2}}\right) ds  \notag \\
&\leq& K \sqrt{t} t^{-C C_{\alpha} } \int_1^t s^{C C_\alpha-2}  ds \leq K\sqrt{t} t^{-1} = K t^{-1/2} .
\end{eqnarray}
Therefore, $ \sqrt{t} \Gamma_t^2 \overset{p} \rightarrow 0$. \textcolor{black}{(Note that $\frac{\partial^3 \bar g}{\partial \theta^3}  (\theta^1_s) Y_s Y_s^{\top}$ is a tensor-matrix product.)}

Next, we address the term $\sqrt{t} \Gamma_t^3$.  We construct the corresponding Poisson equation and use its solution to analyze $\Gamma_t^3$. Define $G(x, \theta) =  \nabla_{\theta} \bar g(\theta) - \nabla_{\theta} g (x, \theta ) $ and let $v(x, \theta)$ be the solution to the PDE $\mathcal{L}_x v(x, \theta) = G(x, \theta)$. Proceeding in a fashion similar to equation (\ref{PoissonErgodic1}) (but with different function $v$ now) we write $\sqrt{t} \Gamma_t^3$:
\begin{eqnarray}
\sqrt{t} \Gamma_t^3 &=& \sqrt{t}  \int_1^t \alpha_s \Phi^{*}_{t,s}  ( \nabla_{\theta} \bar g(\theta_s) - \nabla_{\theta} g (X_s, \theta_s ) ) ds \notag \\
&=& \sqrt{t} \int_1^t \alpha_s  \Phi^{*}_{t,s}  d v_s - \sqrt{t} \int_1^t \alpha_s  \Phi^{*}_{t,s}  \nabla_x v(X_s, \theta_s) d W_s \notag \\
& &-  \sqrt{t}  \int_1^t \alpha_s  \Phi^{*}_{t,s}  \mathcal{L}_{\theta} v (X_s, \theta_s) ds - \sqrt{t} \int_1^t \alpha_s^2  \Phi^{*}_{t,s}  \nabla_{\theta} v (X_{s},\theta_{s})  \nabla_{\theta} f(X_s, \theta_s) dW_{s} \notag \\
& &- \sqrt{t}  \int_1^t \alpha_s^2 \Phi^{*}_{t,s} \nabla_{\theta x} v(X_s, \theta_s) \nabla_{\theta} f(X_s, \theta_s) ds \notag \\
&=& \sqrt{t} \Gamma_t^{3,1} +  \sqrt{t} \Gamma_t^{3,2}+ \sqrt{t} \Gamma_t^{3,3} +  \sqrt{t} \Gamma_t^{3,4} + \sqrt{t} \Gamma_t^{3,5},
\label{Gamma3n}
\end{eqnarray}
where $v(x, \theta)$ satisfies  bounds similar to the ones  in (\ref{BoundL2PDE}).  Following similar steps as in Section \ref{L2proof},
\begin{eqnarray*}
\sqrt{t} \bigg{(} | \Gamma_t^{3,1}| + | \Gamma_t^{3,3}| + | \Gamma_t^{3,4}| + | \Gamma_t^{3,5} | \bigg{)} \overset{L^1} \rightarrow 0,
\end{eqnarray*}
which of course implies that $\sqrt{t} ( \Gamma_t^{3,1} + \Gamma_t^{3,3} + \Gamma_t^{3,4} + \Gamma_t^{3,5} ) \overset{p} \rightarrow 0$.

The term $\sqrt{t} \Gamma_t^{3,2} + \sqrt{t} \Gamma_t^{4}$ is now analyzed.  This term will produce the limiting Gaussian random variable. Recalling the reduction to $\sigma=I$, we have
\begin{eqnarray}
\sqrt{t} \Gamma_t^{3,2} + \sqrt{t} \Gamma_t^{4} = \sqrt{t} \int_1^t \alpha_s  \Phi^{*}_{t,s} \bigg{(}  \nabla_{\theta} f(X_s, \theta_s) -  \nabla_x v(X_s, \theta_s) \bigg{)} d W_s .
\label{NotAMartingale}
\end{eqnarray}

Let $h(x, \theta) =  \bigg{(}  \nabla_{\theta} f(x,\theta) -  \nabla_x v(x,\theta) \bigg{)} \bigg{(}  \nabla_{\theta} f(x,\theta) -  \nabla_x v(x,\theta) \bigg{)}^{\top}$ and $\bar h(\theta) = \int h(x, \theta) \pi(dx)$.  Let $H(x, \theta) = h(x, \theta) - \bar h(\theta)$.  Recall that $\norm{ \nabla_{\theta} f(x, \theta) } \leq K (1 + \norm{x}^q +\norm{\theta})$, $\norm{ \nabla_x v(x, \theta )  } \leq K (1 + \norm{ x }^q +\norm{\theta}^{3})$, $\norm{ \frac{\partial^2 v}{\partial x \partial \theta} (x, \theta )  } \leq K (1 + \norm{ x}^q +\norm{\theta}^{2})$, and $\norm{ \frac{\partial^3 v}{\partial x \partial \theta^2} (x, \theta )  } \leq K (1 + \norm{ x}^q +\norm{\theta} )$.  The latter bounds on $v(x,\theta)$ and its derivatives are from equation (\ref{BoundL2PDE}).  These bounds imply that the function $H(x,\theta)$ and its derivatives have polynomial growth in $\norm{x}$ and $\norm{\theta}$.
Based on Theorem \ref{T:RegularityPoisson}, the solution $w(x, \theta)$ to the PDE $\mathcal{L}_{x} w(x, \theta) = H(x, \theta)$ and its derivatives will also have at most polynomial growth in $\norm{x}$ and $\norm{\theta}$.

We will prove that $\sqrt{t} \Gamma_t^{3,2} + \sqrt{t} \Gamma_t^{4} \overset{d} \rightarrow \mathcal{N}(0,\bar \Sigma )$ for the appropriate limiting variance-covariance matrix $\bar\Sigma$.  The proof will rely upon the Poisson partial differential equation approach using the at most polynomial growth  of its solution and its derivatives together with the uniform boundedness of the moments of $X_{t}$ and $\theta_{t}$ processes to analyze the rate of convergence.

The quadratic covariation matrix of $\sqrt{t} \Gamma_t^{3,2} + \sqrt{t} \Gamma_t^{4}$ is:
\begin{eqnarray*}
\Sigma_t = t \int_1^t \alpha_s^2 \Phi^{*}_{t,s} h(X_s, \theta_s)      \Phi^{*,\top}_{t,s} ds.
\end{eqnarray*}

It is necessary to show that $\Sigma_t \overset{p} \rightarrow \bar \Sigma$ as $t \rightarrow \infty$.  To begin, we show a simpler limit.  Consider the process:
\begin{eqnarray*}
\bar \Sigma_t = t \int_1^t \alpha_s^2 \Phi^{*}_{t,s} \bar h(\theta^{\ast})      \Phi^{*,\top}_{t,s} ds.
\end{eqnarray*}

It will be proven now that $\bar \Sigma_t$ converges to a limit $\bar \Sigma$ as $t \rightarrow \infty$.  Let us now establish the limit $\bar \Sigma$ as $t \rightarrow \infty$. Recall that $\Delta\bar{g}(\theta^{\ast})$ is both strictly positive matrix and symmetric. Therefore, by the eigenvalue decomposition we can write that
\[
\Delta\bar{g}(\theta^{\ast})=U\Lambda U^{\top},
\]
where $U=[u_{1},\cdots,u_{k}]$ is the corresponding matrix of orthogonal eigenvectors $u_1, \ldots, u_k \in \mathbb{R}^k$ and $\Lambda=\text{diag}(\lambda_{1},\cdots,\lambda_{k})$ is the diagonal matrix with the positive eigenvalues $\lambda_{i}>0$ for $i=1,\cdots,k$. Next, we notice that a time transformation as in (\ref{Eq:TransformedPhiMatrix}) allows us to write that for $\tau(t)=s+\int_{s}^{t}\alpha_{u}du$ and $\tilde\Phi^{*}_{\tau(t),s}=\Phi^{*}_{t,s}$
\[
\tilde\Phi^{*}_{\tau,s}=e^{-\Delta\bar{g}(\theta^{\ast})(\tau-s) }.
\]

Combining the latter results we then obtain that
\[
\Phi^{*}_{t,s}=U e^{-\Lambda \int_{s}^{t}\alpha_{u}du }U^{\top}.
\]

For notational convenience, let us now set
\[
\Lambda_{t,s}=e^{-\int_{s}^{t}\alpha_{u}du \Lambda}=\text{diag}\left(e^{-\lambda_{1} \int_{s}^{t}\alpha_{u}du },\cdots, e^{-\lambda_{k} \int_{s}^{t}\alpha_{u}du }\right)=
\text{diag}\left(\left(\frac{s}{t}\right)^{\lambda_{1}C_{\alpha}},\cdots, \left(\frac{s}{t}\right)^{\lambda_{k}C_{\alpha}}\right).
\]

Let us write $\lambda_{(t,s),i}=\left(\frac{s}{t}\right)^{\lambda_{i}C_{\alpha}}$ for $i=1,\cdots,k$. The $(i,j)$-th element of the matrix $\Phi^{*}_{t,s}$ takes the form
\[
\Phi^{*}_{t,s,i,j}=\sum_{m=1}^{k}\lambda_{(t,s),m}u_{i,m}u_{j,m}.
\]

Naturally $(\Phi^{*,\top})_{t,s,k',j}=\Phi^{*}_{t,s,j,k'}=\sum_{m'=1}^{k}\lambda_{(t,s),m'}u_{j,m'}u_{k',m'}$. Hence, the $(i,j)$-th element of the  matrix $\bar \Sigma_t$ takes the form
\begin{align*}
\bar \Sigma_{t,i,j} &= t\int_{1}^{t}\alpha_{s}^{2} \left[\sum_{n} \Phi^{*}_{t,s,i,n}  \sum_{k'=1}^k \bar h(\theta^{\ast})_{n,  k'} (\Phi^{*,\top})_{t,s,k',j}\right]ds\nonumber\\
 &= t\int_{1}^{t}\alpha_{s}^{2} \left[\sum_{n} \left(\sum_{m}\lambda_{(t,s),m}u_{i,m}u_{n,m}\right)  \sum_{k'} \bar h(\theta^{\ast})_{n,  k'} \left(\sum_{m'}\lambda_{(t,s),m'}u_{j,m'}u_{k',m'}\right)\right]ds\nonumber\\
 &=  \sum_{n}\sum_{m} \sum_{k'} \sum_{m'} t\int_{1}^{t}\alpha_{s}^{2}  \left[\left(\lambda_{(t,s),m}u_{i,m}u_{n,m}\right)   \bar h(\theta^{\ast})_{n,  k'} \left(\lambda_{(t,s),m'}u_{j,m'}u_{k',m'}\right)\right]ds\nonumber\\
 &=  \sum_{n}\sum_{m} \sum_{k'} \sum_{m'} \left[ t\int_{1}^{t}\alpha_{s}^{2}  \left(\frac{s}{t}\right)^{\left(\lambda_{m}+\lambda_{m'}\right)C_{\alpha}} ds\right]   u_{i,m}u_{n,m}\bar h(\theta^{\ast})_{n,  k'} u_{j,m'}u_{k',m'}  \nonumber\\
 &=  \sum_{n}\sum_{m} \sum_{k'} \sum_{m'} \left[ \frac{C_{\alpha}^{2}}{\left(\lambda_{m}+\lambda_{m'}\right)C_{\alpha}-1}\left(1-t^{1-\left(\lambda_{m}+\lambda_{m'}\right)C_{\alpha}}\right)\right]   u_{i,m}u_{n,m}\bar h(\theta^{\ast})_{n,  k'} u_{j,m'}u_{k',m'}.  \nonumber
\end{align*}

%

By the Rayleigh-Ritz theorem, $\lambda_1, \ldots, \lambda_k \geq C$.  Therefore, $C_{\alpha} ( \lambda_m + \lambda_{k'} ) \geq 2 C_{\alpha} C > 1$.  Then, we get that the $(i,j)$-th element of the limiting quadratic covariation matrix $\bar \Sigma$ takes the form
\begin{align}
\bar \Sigma_{i,j} &= \lim_{t \rightarrow \infty} \bar \Sigma_{t,i,j} = \sum_{n}\sum_{m} \sum_{k'} \sum_{m'} \left[ \frac{C_{\alpha}^{2}}{\left(\lambda_{m}+\lambda_{m'}\right)C_{\alpha}-1}\right]   u_{i,m}u_{n,m}\bar h(\theta^{\ast})_{n,  k'} u_{j,m'}u_{k',m'}\nonumber\\
&= \sum_{n}\sum_{m} u_{i,m}u_{n,m}\sum_{k'} \bar h(\theta^{\ast})_{n,  k'} \sum_{m'} \left[ \frac{C_{\alpha}^{2}}{\left(\lambda_{m}+\lambda_{m'}\right)C_{\alpha}-1}\right] u_{j,m'}u_{k',m'}.\label{Eq:LimitingSigma}
\end{align}

Notice that we can conveniently write that
\begin{align}
\bar{\Sigma}=C_{\alpha}^{2}\int_{0}^{\infty}e^{-s\left(C_{\alpha}\Delta\bar{g}(\theta^{\ast})-I\right)}\bar{h}(\theta^{\ast})e^{-s\left(C_{\alpha}\left(\Delta\bar{g}\right)^{\top}(\theta^{\ast})-I\right)}ds.\label{Eq:LimitingSigma2}
\end{align}

It remains to show that $\mathbb{E} \norm{\Sigma_t - \bar \Sigma_t }_1  \rightarrow 0$ as $t \rightarrow \infty$.  If this is true, then the triangle inequality yields
\begin{eqnarray*}
\mathbb{E} \norm{ \Sigma_t - \bar \Sigma }_1 \leq \mathbb{E} \norm{ \Sigma_t - \bar \Sigma_t }_1+ \mathbb{E} \norm{ \bar \Sigma_t - \bar \Sigma }_1 \rightarrow 0.
\end{eqnarray*}
Then, $ \mathbb{E} \norm{ \Sigma_t - \bar \Sigma }_1 \rightarrow 0$ would imply that $\Sigma_t \overset{p} \rightarrow \bar \Sigma$.

To prove that $\mathbb{E} \norm{\Sigma_t - \bar \Sigma_t }_1  \rightarrow 0$, we begin by defining
\begin{eqnarray*}
\bar V_t = t \int_1^t \alpha_s^2 \Phi^{*}_{t,s} \bar h(\theta_s)      \Phi^{*,\top}_{t,s} ds.
\end{eqnarray*}

By the triangle inequality, $\norm{ \Sigma_t - \bar \Sigma_t }_1 \leq \norm{  \Sigma_t - \bar V_t }_1 + \norm{ \bar V_t - \bar \Sigma_t }_1 $.  We first address the second term:
\begin{eqnarray}
 \norm{ \bar V_t - \bar \Sigma_t }_1 &=& \norm{ t \int_1^t \alpha_s^2 \big{(} \Phi^{*}_{t,s} \bar h(\theta_s) \Phi^{*,\top}_{t,s} -  \Phi^{*}_{t,s} \bar h(\theta^{\ast}) \Phi^{*,\top}_{t,s} \big{)} ds}_1.
 \label{TriangleCLT}
  \end{eqnarray}

  The $(i,j)$-th element of the matrix $\bar V_t - \bar \Sigma_t$ is:
  \begin{eqnarray*}
 \big{(}  \bar V_t - \bar \Sigma_t  \big{)}_{i,j} &=&  t \int_1^t \alpha_s^2   \bigg{(} \sum_{n} \Phi^{*}_{t,s,i,n}  \sum_{k'=1}^k \bar h(\theta_s)_{n,  k'} \Phi^{*,\top}_{t,s,k',j} -  \sum_{n} \Phi^{*}_{t,s,i,n}  \sum_{k'=1}^k \bar h(\theta^{\ast})_{n,  k'} \Phi^{*,\top}_{t,s,k',j} \bigg{)} ds   \notag \\
 &=&  t \int_1^t \alpha_s^2   \sum_{n} \Phi^{*}_{t,s,i,n}  \sum_{k'=1}^k \nabla_{\theta} \bar h(\theta_s^1)_{n,  k'}^{\top} ( \theta_s - \theta^{\ast}) \Phi^{*,\top}_{t,s,k',j} ds,
  \end{eqnarray*}
  where $\theta_s^1$ is an appropriately chosen point in the segment connecting  $\theta_s$ and $\theta^{\ast}$.  Recall now that $v$ and its derivatives can grow at most polynomially in $\norm{x}$ and $\norm{\theta}$. 
  As a result of the specific growth rates we have
\begin{eqnarray*}
\norm{ \nabla_{\theta} \bar{h}(\theta) }  \leq K (1 + \norm{ \theta }^{6} ) .
\end{eqnarray*}
In addition, $\mathbb{E}[  \norm{\theta_s - \theta^{\ast} }^2 ] \leq \frac{K}{s}$ from the convergence rate in Section \ref{L2proof}.  Using these facts, the uniform in time moments bound on $\theta_t$, and the Cauchy Schwartz inequality,
\begin{eqnarray*}
\mathbb{E}[  | \big{(}  \bar V_t - \bar \Sigma_t  \big{)}_{i,j} | ] & \leq & t \int_1^t \alpha_s^2 \mathbb{E} \bigg{|} \sum_{n} \Phi^{*}_{t,s,i,n}  \sum_{k'=1}^k \nabla_{\theta} \bar h(\theta_s^1)_{n,  k'}^{\top} ( \theta_s - \theta^{\ast}) \Phi^{*,\top}_{t,s,k',j} \bigg{|} ds \notag \\
&\leq& K t \int_1^t \alpha_s^2   \sum_{n}  \norm{ \Phi^{*}_{t,s,i,n} } \sum_{k'=1}^k \mathbb{E} \bigg{[}  \norm{ \nabla_{\theta} \bar h(\theta_s^1)_{n,  k'} }^2  \bigg{]}^{\frac{1}{2}} \mathbb{E}\bigg{[} \norm{ \theta_s - \theta^{\ast}}^2  \bigg{]}^{\frac{1}{2}} \norm{\Phi^{*,\top}_{t,s,k',j} } ds \notag \\
&\leq& K t \int_1^t s^{- \frac{5}{2}}  \norm{ \Phi^{*}_{t,s} }^2 ds.
\end{eqnarray*}
Therefore, $ \norm{ \bar V_t - \bar \Sigma_t }_1 \rightarrow 0$ as $t \rightarrow \infty$.

Now, let's address $ \norm{  \Sigma_t - \bar V_t }_1$ using the Poisson equation method.
\begin{eqnarray*}
\Sigma_t - \bar V_t &=& t \int_1^t \alpha_s^2 \bigg{(} \Phi^{*}_{t,s} h(\theta_s, X_s)    \Phi^{*,\top}_{t,s}  - \Phi^{*}_{t,s} \bar h(\theta_s)    \Phi^{*,\top}_{t,s}   \bigg{)} ds. \notag
\end{eqnarray*}

 The $(i,j)$-th element of the matrix $\Sigma_t - \bar V_t$ is:
  \begin{eqnarray*}
 \big{(}  \bar V_t - \bar \Sigma_t  \big{)}_{i,j} &=&  t \int_1^t \alpha_s^2   \sum_{n} \Phi^{*}_{t,s,i,n}  \sum_{k'=1}^k \bigg{(} h(\theta_s, X_s)_{n,  k'} -  \bar h(\theta_s)_{n,  k'} \bigg{)}\Phi^{*,\top}_{t,s,k',j}ds.   \notag
 \end{eqnarray*}

Using It\^{o}'s formula, the bounds (\ref{RegularityBounds1}) on the Poisson equation $\mathcal{L}_{x} w = H$, the moment bounds on $X_t$ and $\theta_t$, and the  It\^{o} Isometry, it can be shown that
\begin{eqnarray*}
 \mathbb{E} \big{|}  \big{(}  \bar V_t - \bar \Sigma_t  \big{)}_{i,j}  \big{|} \rightarrow 0,
\end{eqnarray*}
as $t \rightarrow \infty$.  This of course implies that $\mathbb{E} \norm{ \bar V_t - \bar \Sigma_t }_1  \rightarrow 0$.  Combining results and using the triangle inequality, we have the desired result $\Sigma_t  \overset{p} \rightarrow  \bar \Sigma$ as $t \rightarrow \infty$.

The convergence in probability of the quadratic variation $\Sigma_t$ for equation (\ref{NotAMartingale}) implies that (\ref{NotAMartingale}) converges in distribution to a mean zero normal random variable with covariance $\bar \Sigma$ (see Section 1.2.2  in \cite{Kutoyants}).  Combining all of the results yields the central limit theorem:
\begin{eqnarray*}
\sqrt{t} Y_t \overset{d} \rightarrow \mathcal{N}(0,\bar \Sigma ).
\end{eqnarray*}

\section{Proof of Central Limit Theorem in the non-convex case} \label{QuasiConvexproof}



From Theorem \ref{T:AlternateMainTheoremSGD1}, we know that $ \norm{ \nabla \bar g (\theta_t) } \rightarrow 0$ almost surely.  Under the imposed conditions this implies that either $\norm{ \theta_t - \theta^{\ast} } \rightarrow 0$ or $\norm{ \theta_t } \rightarrow \infty$.  We must therefore first show that $\norm{ \theta_t }$ remains finite almost surely.

The parameter $\theta_t$ evolves as:
\begin{eqnarray}
\theta_t =  \theta_{1}+\int_1^t -\alpha_s  \nabla_{\theta} \bar g ( \theta_s ) ds +  \int_1^t \alpha_s  \big{[} \nabla_{\theta} \bar g ( \theta_s ) - \nabla_{\theta} g ( \theta_s, X_s) \big{]} ds +  \int_{1}^{t}\alpha_s  \nabla_{\theta} f(X_s, \theta_s) d W_s.
\label{thetaQCcase}
\end{eqnarray}

Recall from Condition \ref{A:Assumption2} that $\nabla_{\theta_i} \bar g(\theta) > 0$ if $\theta_i - \theta^{\ast}_i > R_0$ and $\nabla_{\theta_i} \bar g(\theta) < 0$ if $\theta_i - \theta^{\ast}_i < -R_0$ for $i =1, 2, \ldots, k$.  If $|\theta_{t,i} (\omega) | \rightarrow \infty$, then either $\theta_{t,i}(\omega) \rightarrow + \infty$ or $\theta_{t,i}(\omega) \rightarrow - \infty$ since $\theta_t$ has continuous paths (i.e., a divergent sequence such as $-2^n$ cannot occur).  Next, note that the second and third integrals in equation (\ref{thetaQCcase}) converge to finite random variables almost surely.  Suppose that $\theta_{t,i}(\omega) \rightarrow + \infty$.  This implies that there exists a $T(\omega)$ such that $\theta_{t,i} > \theta^{\ast}_i$ for all $t \geq T(\omega)$.  However, $\int_{T(\omega)}^t -\alpha_s  \nabla_{\theta_i} \bar g ( \theta_s ) ds < 0$.  This, combined with the fact that the second and third terms in (\ref{thetaQCcase}) converge to finite random variables almost surely, proves that $\theta_{t,i} < + \infty$ with probability one.  A similar argument shows that $\theta_{t,i} > - \infty$ with probability one.  Therefore, $| \theta_{t,i} - \theta^{\ast}_i | \rightarrow 0$ almost surely.

$\theta_t \overset{a.s.} \rightarrow \theta^{\ast}$ implies that, for any $\delta > 0$ sufficiently small, there exists a (random) time $\tau_{\delta}$ such that $\norm{ \theta_t - \theta^{\ast} } < \delta $ for all $t \geq \tau_{\delta}$. Notice that for fixed $\delta>0$, $\tau_{\delta}$ is almost surely finite, but we also remark that $\tau_{\delta}$ is not a stopping time. In this section, the parameter $\delta$ is chosen sufficiently small such that $0<\delta<\delta^{*}$, where $\delta^{*}$ is taken from Condition \ref{A:Assumption2}. Using a Taylor expansion:
%
\begin{eqnarray*}
\nabla_{\theta} \bar g (\theta_t) &=&  \nabla_{\theta} \bar g (\theta^{\ast} ) + \Delta \bar g (\theta_t^1) ( \theta_t - \theta^{\ast} )= \Delta \bar g (\theta_t^1) ( \theta_t - \theta^{\ast} ),
\end{eqnarray*}
where $\theta_t^1$ is appropriately chosen in the segment connecting  $\theta_t$ to $\theta^{\ast}$. Then we have
\begin{eqnarray*}
d  ( \theta_t  - \theta^{\ast} )&=& - \alpha_t \Delta \bar g( \theta_t^1) ( \theta_t  - \theta^{\ast})    dt   + \alpha_t ( \nabla_{\theta} \bar g(\theta_t) - \nabla_{\theta} g (X_t, \theta_t ) ) dt + \alpha_t \nabla_{\theta} f(X_t, \theta_t) d W_t.
\end{eqnarray*}

Let $\Phi_{t,s}$ satisfy
\begin{eqnarray}
d \Phi_{t,s} &=& - \alpha_t \Delta \bar g( \theta_t^1) \Phi_{t,s} dt, \qquad
\Phi_{s,s} = I,\notag
\end{eqnarray}
and notice that for $t > s > \tau_{\delta}$,
\begin{eqnarray*}
\norm{  \Phi_{t,s} }^2  \leq K t^{-2 C  C_{\alpha} } s^{ 2 C C_{\alpha } }.
\end{eqnarray*}
We will also later make use of the fact that $\Phi_{t,s}$ satisfies the semi-group property $\Phi_{t,s}= \Phi_{t,\tau} \Phi_{\tau,s}$ (for reference, see Proposition 2.14 in \cite{Chicone}). Letting $Y_t = \theta_t - \theta^{\ast}$, we obtain
\begin{eqnarray}
Y_t &=&  \Phi_{t,1} Y_1    + \int_1^t \alpha_s \Phi_{t,s}  ( \nabla_{\theta} \bar g(\theta_s) - \nabla_{\theta} g (X_s, \theta_s ) ) ds + \int_1^t   \alpha_s \Phi_{t,s}    \nabla_{\theta} f(X_s, \theta_s) d W_s \notag \\
&=& \Gamma_t^1 + \Gamma_t^2 + \Gamma_t^3.\notag
\end{eqnarray}

The first term $\Gamma_t^1$ is analyzed below.
\begin{lemma}
\begin{eqnarray}
\sqrt{t}  \Gamma_t^1 \overset{a.s.} \rightarrow 0,\notag
\end{eqnarray}
as $t \rightarrow \infty$.
\end{lemma}
\begin{proof}
For $t\geq\tau_{\delta}$ we have
\begin{eqnarray}
\norm{ \sqrt{t}  \Gamma_t^1 }_1 = \sqrt{t} \norm{  \Phi_{t, \tau_{\bar \delta} } \Phi_{\tau_{\bar \delta}, 1}  Y_1 }_1
& \leq & K \sqrt{t} \norm{ \Phi_{t, \tau_{\bar \delta} } }  \norm{\Phi_{\tau_{\bar \delta}, 1} Y_1 } \notag \\
&\leq & K  t^{-C C_{\alpha} + \frac{1}{2} } \tau_{\bar \delta}^{ C C_{\alpha} }  \norm{\Phi_{\tau_{\bar \delta}, 1} Y_1 }   \notag \\
&=& K(\tau_{\bar \delta} ) t^{-C C_{\alpha} + \frac{1}{2} }.\notag
\end{eqnarray}
where $K(\tau_{\bar \delta} ) = \tau_{\bar \delta}^{ C C_{\alpha} }  \norm{\Phi_{\tau_{\bar \delta}, 1} Y_1 }$ is almost surely finite since $\mathbb{P} [ \tau_{\bar \delta} < \infty] = 1$.  Then, since $C C_{\alpha} >1$, $\sqrt{t}  \Gamma_t^1 \overset{a.s.} \rightarrow 0$ as $t \rightarrow \infty$.
\end{proof}

Let $v(x, \theta)$ satisfy the Poisson PDE $\mathcal{L}_x v = G(x, \theta)$ where $G(x, \theta) = \nabla_{\theta} \bar g(\theta) - \nabla_{\theta} g (x, \theta )$.  The solution $v(t,x)$ and its relevant partial derivatives will be growing at most polynomially in $\norm{x}$ and linearly in $\norm{\theta}$ due to the assumptions of Theorem \ref{T:MainTheorem3}. It\^{o}'s formula yields the representation
\begin{eqnarray}
\Gamma_t^2 &=&  \int_1^t \alpha_s  \Phi_{t,s}  ( \nabla_{\theta} \bar g(\theta_s) - \nabla_{\theta} g (X_s, \theta_s ) ) ds \notag \\
&=&  \int_1^t \alpha_s \Phi_{t,s}  d v_s - \int_1^t \alpha_s  \Phi_{t,s}  \nabla_x v(X_s, \theta_s) d W_s \notag \\
& &-  \int_1^t \alpha_s \Phi_{t,s}  \mathcal{L}_{\theta} v (X_s, \theta_s) ds -  \int_1^t \alpha_s^2  \Phi_{t,s}  \nabla_{\theta} v (X_{s},\theta_{s})  \nabla_{\theta} f(X_s, \theta_s) dW_{s} \notag \\
& &-  \int_1^t \alpha_s^2 \Phi_{t,s} \nabla_{\theta x} v(X_s, \theta_s) \nabla_{\theta} f(X_s, \theta_s) ds \notag \\
&=& \Gamma_t^{2,1} + \Gamma_t^{2,2}+ \Gamma_t^{2,3} + \Gamma_t^{2,4} + \Gamma_t^{2,5}.\notag
\end{eqnarray}

In order now to analyze the terms involved in \textcolor{black}{$\Gamma_{t}^{2,2}$ and $\Gamma_{t}^{3}$} we need some intermediate results that we state now below. For presentation purposes the proof of these Lemmas is deferred to the end of this section.

\begin{lemma}\label{L:MainTermsConvergenceZero}
Let $\zeta(x, \theta)$ be a (potentially matrix-valued) function that can grow at most polynomially in $x$ and $\theta$.  Then, we have that
\begin{enumerate}
\item{
$I_{t}^{1}=\sqrt{t} \int_1^t \alpha_s^2 \Phi_{t,s} \zeta(X_s, \theta_s) ds \overset{a.s.} \rightarrow 0$, as $t \rightarrow \infty$, and}
\item{$I_{t}^{2}=\sqrt{t} \int_1^t \alpha_s^2 \Phi_{t,s} \zeta(X_s, \theta_s) dW_{s} \rightarrow 0$, in probability, as $t \rightarrow \infty$.}
\end{enumerate}
\end{lemma}

Lemma \ref{L:MainTermsConvergenceZero} immediately gives that the terms $\sqrt{t} \Gamma_t^{2,3}$ (recall that the operator $\mathcal{L}_{\theta}$ has a multiplicative factor of $\alpha_{t}$), $\sqrt{t} \Gamma_t^{2,5}$ converge almost surely to zero and that $\sqrt{t} \Gamma_t^{2,4}$ converges in probability to zero, as $t\rightarrow\infty$. To see this it is enough to notice that all of these terms take the form of the $I_{t}^{1}$ and $I_{t}^{2}$ quantities mentioned in Lemma \ref{L:MainTermsConvergenceZero}. The term $\sqrt{t} \Gamma_{t}^{2,1}$ also converges almost surely to zero by first rewriting using It\^{o} formula, as it was done for the corresponding term of Section \ref{L2proof} (refer to (\ref{Eq:Gamma21term}) with $p=2$ and replace $\Psi_{t,s}^{(p)}$ by $\Phi_{t,s}$) and then using again Lemma \ref{L:MainTermsConvergenceZero}.

The limiting Gaussian random variable will be produced by $\Gamma_t^{2,2}$ and $\Gamma_t^3$.  Therefore, it remains to analyze $\sqrt{t} ( \Gamma_t^{2,2} + \Gamma_t^3 )$. We have
\begin{eqnarray*}
 \sqrt{t} ( \Gamma_t^{2,2} + \Gamma_t^3 ) = \sqrt{t} \int_1^t \alpha_s  \Phi_{t,s} \bigg{(}  \nabla_{\theta} f(X_s, \theta_s) -  \nabla_x v(X_s, \theta_s) \bigg{)} d W_s .
\end{eqnarray*}

Using the results from \cite{Kutoyants}, it is sufficient to prove the convergence in probability of
\begin{eqnarray*}
\Sigma_t = t \int_1^t \alpha_s^2 \Phi_{t,s} h(X_s, \theta_s)      \Phi_{t,s}^{\top} ds
\end{eqnarray*}
to a deterministic quantity. We recall here that $h(x,\theta)=\bigg{(}  \nabla_{\theta} f(x, \theta) -  \nabla_x v(x, \theta) \bigg{)}\bigg{(}  \nabla_{\theta} f(x, \theta) -  \nabla_x v(x, \theta) \bigg{)}^{\top}$. As before, let $\Phi_{t,s}^{\ast}$ be the solution to
\begin{eqnarray}
d \Phi_{t,s}^{\ast} &=& - \alpha_t \Delta \bar g( \theta^{\ast}) \Phi_{t,s}^{\ast} dt, \qquad
\Phi_{s,s}^{\ast} = I. \notag
\end{eqnarray}
Recall that $\Phi_{t,s}^{\ast}$ satisfies the bound
\begin{eqnarray*}
\norm{  \Phi_{t,s}^{\ast} }^2  \leq K t^{-2 C  C_{\alpha} } s^{ 2 C C_{\alpha } }.
\end{eqnarray*}

Define $\bar \Sigma_t$ and $\bar \Sigma_t^{\ast}$ as:
\begin{eqnarray}
\bar \Sigma_t &=& t \int_1^t \alpha_s^2 \Phi_{t,s} \bar h(\theta^{\ast} )      \Phi_{t,s}^{\top} ds, \text{ and }
\bar \Sigma^{\ast}_t =  t \int_1^t \alpha_s^2 \Phi_{t,s}^{\ast} \bar h(\theta^{\ast} )      \Phi_{t,s}^{\ast, \top} ds.\notag
\end{eqnarray}

Note that we already proved in the previous section that $\bar \Sigma^{\ast}_{t}$ converges in probability to $\bar{\Sigma}$ as $t\rightarrow\infty$.  We would like to show that $\norm{\bar \Sigma_t - \bar \Sigma^{\ast}_{t} } \rightarrow 0$ as $t \rightarrow \infty$, almost surely.

\begin{lemma} \label{NonConvexLemmabarSigmaAst}
$ \norm{\bar \Sigma_t  - \bar \Sigma^{\ast}_{t} } \rightarrow 0$, with probability 1, as $t \rightarrow \infty$.
\end{lemma}

To prove that $\norm{\Sigma_t - \bar \Sigma_t }  \rightarrow 0$, we begin by defining
\begin{eqnarray*}
\bar V_t = t \int_1^t \alpha_s^2 \Phi_{t,s} \bar h(\theta_s)      \Phi_{t,s}^{\top} ds
\end{eqnarray*}
By the triangle inequality, $\norm{ \Sigma_t - \bar \Sigma_t } \leq \norm{  \Sigma_t - \bar V_t }+ \norm{ \bar V_t - \bar \Sigma_t } $.  We then have the following lemmas.

\begin{lemma} \label{NonConvexLemmabarVbarSigma}
$ \norm{ \bar V_t - \bar \Sigma_t }  \rightarrow 0$, with probability 1, as $t \rightarrow \infty$.
\end{lemma}

\begin{lemma} \label{NonConvexLemmaSigmabarV}
$\norm{  \Sigma_t - \bar V_t } \rightarrow 0$, in probability, as $t \rightarrow \infty$.
\end{lemma}

By the triangle inequality,
\begin{eqnarray}
\norm{ \Sigma_t - \bar \Sigma} &\leq& \norm{  \Sigma_t - \bar V_t }+ \norm{ \bar V_t - \bar \Sigma_t } + \norm{\bar \Sigma_t - \bar \Sigma}.\notag
\end{eqnarray}

Combining Lemmas \ref{NonConvexLemmaSigmabarV}, \ref{NonConvexLemmabarSigmaAst}, and \ref{NonConvexLemmabarVbarSigma}, $\norm{ \Sigma_t - \bar \Sigma} \overset{p} \rightarrow 0$ as $t \rightarrow \infty$.  Therefore, using the results from \cite{Kutoyants},
\begin{eqnarray}
 \sqrt{t} ( \Gamma_t^{2,2} + \Gamma_t^3 ) \overset{d} \rightarrow \mathcal{N}(0, \bar \Sigma),\notag
\end{eqnarray}
as $t \rightarrow \infty$.

%
%
%
%

Combining all of the results yields the central limit theorem:
\begin{eqnarray*}
\sqrt{t} Y_t \overset{d} \rightarrow \mathcal{N}(0,\bar \Sigma ),
\end{eqnarray*}
which is our desired result.

\subsection{Proof of Lemmas \ref{L:MainTermsConvergenceZero}-\ref{NonConvexLemmaSigmabarV}}
In this subsection we give the proof of the lemmas that were used in the proof of the central limit theorem of the non-convex case.

First we need an intermediate result to properly handle convergence to zero of multidimensional stochastic integrals. Such a result should be standard in the literature, but because we did not manage to locate an exact statement we present Lemma \ref{KutoyantsExtensionLemma}.
\begin{lemma} \label{KutoyantsExtensionLemma}
Let $Z_t = \int_1^t b(t, s,X_s, \theta_s) d W_s$.  Let $p \in \mathbb{N}$ be a given integer, and consider a constant $c>p-\frac{1}{2}$ and a matrix $E$ where $E E^{\top}$ is positive definite,  such that
\begin{eqnarray}
 t^{-c+p-\frac{1}{2} } \int_1^t \alpha_s^{p} s^{c} b(t, s,X_s, \theta_s) E^{\top} ds \overset{p} \rightarrow 0,
 \end{eqnarray}
as $t \rightarrow \infty$.  Then, if $ \int_1^t b(t,s, X_s, \theta_s) b(t, s,X_s, \theta_s)^{\top} ds \overset{p} \rightarrow 0$ as $t \rightarrow \infty$, we have that  $Z_t \overset{p} \rightarrow 0$ as $t \rightarrow \infty$.
\end{lemma}

\begin{proof}[Proof of Lemma \ref{KutoyantsExtensionLemma}]
Let $\eta>0$ be arbitrarily chosen and construct the random variable
\[
\tilde Z_t =\eta \sqrt{\frac{2c-2p+1}{C_{\alpha}^{2p}}}  t^{-c+p-\frac{1}{2}} \int_1^t  \alpha_s^{p} s^c E d W_s.
\]

From Proposition 1.21 in Section 1.2.2  of \cite{Kutoyants}, $Z_t + \tilde Z_t \overset{d} \rightarrow \mathcal{N}(0, \eta^2 E E^{\top} )$ as $t \rightarrow \infty$.  In addition, $ \tilde Z_t \overset{d} \rightarrow \mathcal{N}(0, \eta^2 E E^{\top} )$ as $t \rightarrow \infty$.  Then, from the triangle inequality,
\begin{eqnarray}
& &\mathbb{P} \bigg{[} \norm{ Z_t }^2 > \epsilon  \bigg{]} \leq  \sum_{i=1}^k \mathbb{P} \bigg{[} Z_{t,i}^2 >  \frac{\epsilon}{k} \bigg{]} = \sum_{i=1}^k \mathbb{P} \bigg{[}  | Z_{t,i} | > \sqrt{ \frac{\epsilon}{k} }    \bigg{]} =  \sum_{i=1}^k \mathbb{P} \bigg{[}  | Z_{t,i} + \tilde Z_{t,i} - \tilde Z_{t,i} | > \sqrt{ \frac{\epsilon}{k} }    \bigg{]}   \notag \\
&\leq&  \sum_{i=1}^k \mathbb{P} \bigg{[}  | \tilde Z_{t,i} | + | Z_{t,i} + \tilde Z_{t,i} | > \sqrt{ \frac{\epsilon}{k} }  \bigg{]}   \notag \\
&\leq&  \sum_{i=1}^k \bigg{(} \mathbb{P} \bigg{[}  | \tilde Z_{t,i} |  > \frac{1}{2} \sqrt{ \frac{\epsilon}{k} }   \bigg{]} +  \mathbb{P} \bigg{[}  |  Z_{t,i} + \tilde Z_{t,i} |  > \frac{1}{2} \sqrt{ \frac{\epsilon}{k} }  \bigg{]} \bigg{)}.
\label{tildeZinequality}
\end{eqnarray}
For each fixed $\eta>0$, the RHS of the inequality (\ref{tildeZinequality}) converges to a finite quantity as $t\rightarrow\infty$ due to the continuous mapping theorem and the convergence in distribution of $ Z_{t,i} + \tilde Z_{t,i}$ and $\tilde Z_t$.  Furthermore, the limit of the RHS can be made arbitrarily small by choosing a sufficiently small $\eta$.  Therefore, for any $\delta > 0$, there exists a $\eta > 0$ such that:
\begin{eqnarray*}
\textcolor{black}{\lim \sup_{t \rightarrow \infty} } \mathbb{P} \bigg{[} \norm{ Z_t }^2 > \epsilon  \bigg{]} < \delta.
\end{eqnarray*}
\end{proof}
\begin{proof}[Proof of Lemma \ref{L:MainTermsConvergenceZero}]
Let us first prove the first statement of the lemma. Without loss of generality, let $t \geq \tau_{\delta}$. To begin, divide $[0,t]$ into two regimes $[1,  \tau_{\delta}\wedge t ]$ and $[\tau_{\delta}\wedge t, \tau_{\delta}\vee t ]$:
\begin{align}
I_{t}^{1}&=\sqrt{t} \int_1^t \alpha_s^2 \Phi_{t,s} \zeta(X_s, \theta_s) ds  = \sqrt{t} \int_1^{ \tau_{\delta}\wedge t } \alpha_s^2 \Phi_{t,s} \zeta(X_s, \theta_s) ds  + \sqrt{t} \int_{\tau_{\delta}\wedge t}^{\tau_{\delta}\vee t} \alpha_s^2 \Phi_{t,s} \zeta(X_s, \theta_s) ds.\nonumber\\
&=I_{t}^{1,1}+I_{t}^{1,2}.
\label{Riemann1}
\end{align}
Let us first study the second term.
\begin{align}
\norm{I_{t}^{1,2}}_1 &= \norm{ \sqrt{t} \int_{\tau_{\delta}\wedge t}^{\tau_{\delta}\vee t} \alpha_s^2 \Phi_{t,s} \zeta(X_s, \theta_s) ds }_1 \leq K \sqrt{t} \int_{\tau_{\delta}\wedge t}^{\tau_{\delta}\vee t} \alpha_s^2  \norm{ \Phi_{t,s}} \norm{ \zeta(X_s, \theta_s) } ds \notag \\
&\leq K t^{- C C_{\alpha} +1/2} \int_{\tau_{\delta}\wedge t}^{\tau_{\delta}\vee t} s^{ C C_{\alpha} - 2} \left(1 + \norm{X_s}^{m_{1}}+\norm{\theta_s}^{m_{2}}\right) ds  \notag \\
&\leq K t^{- C C_{\alpha} +1/2} \int_{1}^{\tau_{\delta}\vee t} s^{ C C_{\alpha} - 2} \left(1 + \norm{X_s}^{m_{1}}+\norm{\theta_s}^{m_{2}}\right) ds.\notag
\end{align}

Let us now define the quantity
\begin{align*}
L_{t}&=K t^{- C C_{\alpha} +1/2} \int_{1}^{t} s^{ C C_{\alpha} - 2} \left(1 + \norm{X_s}^{m_{1}}+\norm{\theta_s}^{m_{2}}\right) ds,
\end{align*}
and notice that with probability one we have
\begin{eqnarray*}
 \limsup_{t \rightarrow \infty} \norm{I_t^{1,2}}_1 \leq \limsup_{t \rightarrow \infty} L_t.
\end{eqnarray*}

For $\varepsilon>0$, consider now the event $A_{t,\varepsilon}=\left\{L_{t}\geq t^{\varepsilon-1/2}\right\}$. Using the uniform in time bounds for the moments of $X_{s}$ and $\theta_{s}$ we obtain that
\begin{align*}
\mathbb{E} \left| L_{t}\right|&\leq Kt^{- C C_{\alpha} +1/2} \int_{1}^{t} s^{ C C_{\alpha} - 2}ds\leq K\left(t^{-1/2}-t^{-CC_{\alpha}+1/2}\right).
\end{align*}

Then, Markov's inequality and the fact that $CC_{\alpha}>1$ give
\begin{align*}
\mathbb{P}\left[A_{t,\varepsilon}\right]&\leq\frac{\mathbb{E}\left|L_{t}\right|}{t^{\varepsilon-1/2}}\leq  Kt^{-\varepsilon}.
\end{align*}

The latter then implies that
\begin{eqnarray*}
\sum_{n \in \mathbb{N} } \mathbb{P}\left[ A_{2^n,\varepsilon} \right]  < \infty,
\end{eqnarray*}
 which then due to the Borel-Cantelli lemma, guarantees the existence of a finite positive random variable $d(\omega)$ and of some $n_0 < \infty$ such that for every $n \geq n_0$,
\begin{eqnarray*}
L_{2^n} \leq d( \omega) 2^{- n\varepsilon}.
\end{eqnarray*}

For any $t \in [2^n, 2^{n+1} ]$ and $n \geq n_0$,
\begin{eqnarray}
L_{t} & \leq & K 2^{n (- C C_{\alpha}  + 1/2)}  \int_{1}^{2^{n+1} } s^{ C C_{\alpha}  -2} (1 + \norm{X_s}^{m_{1}}+\norm{\theta_{s}}^{m_{2}})    ds. \notag \\
&\leq & K 2^{C C_{\alpha}  - 1/2}  L_{2^{n+1} }  \notag \\
&\leq& K 2^{C C_{\alpha}  - 1/2}   d( \omega) 2^{- (n+1)\varepsilon } \notag \\
&\leq& K(\omega)  t^{-\varepsilon }.\notag
\end{eqnarray}

Therefore, $L_t \overset{a.s.} \rightarrow 0$, which immediately implies that $I_t^{1,2} \overset{a.s.} \rightarrow 0$. Next, we analyze the first term on the RHS of equation (\ref{Riemann1}).  If $t > \tau_{\delta}$, the semi-group property yields:
\begin{eqnarray*}
 I_{t}^{1,1}&=\sqrt{t} \int_1^{ \tau_{\delta}\wedge t } \alpha_s^2 \Phi_{t,s} \zeta(X_s, \theta_s) ds  =  \sqrt{t}\Phi_{t,\tau_{\delta} }  \int_1^{ \tau_{\delta}\wedge t } \alpha_s^2  \Phi_{\tau_{\delta},s} \zeta(X_s, \theta_s) ds.
 \end{eqnarray*}

 Therefore, if $t > \tau_{\delta}$,
 \begin{align}
 \norm{ \sqrt{t} \int_1^{ \tau_{\delta}\wedge t } \alpha_s^2 \Phi_{t,s} \zeta(X_s, \theta_s) ds } &\leq   K \sqrt{t}  \norm{ \Phi_{t,\tau_{\delta} } }  \norm{ \int_1^{ \tau_{\delta}} \alpha_s^2  \Phi_{\tau_{\delta},s} \zeta(X_s, \theta_s) ds } \leq C(\tau_{\delta} ) t^{- C C_{\alpha} + 1/2}.\notag
 \end{align}

 The constant $C(\tau_{\delta})$ is almost surely finite since $\mathbb{P} [ \tau_{\delta} < \infty] = 1$ and because $\norm{X_{s}}$ and $\norm{\theta_{s}}$ are almost surely finite for $s\leq \tau_{\delta}$.  Therefore, using the constraint $CC_{\alpha}>1$, we have obtained
 \begin{eqnarray}
 \norm{ \sqrt{t} \int_1^{ \tau_{\delta}\wedge t } \alpha_s^2 \Phi_{t,s} \zeta(X_s, \theta_s) ds } \overset{a.s.} \rightarrow 0.\notag
 \end{eqnarray}

 Combining results, the integral (\ref{Riemann1}) converges to $0$ almost surely as $t \rightarrow \infty$.

Next, we prove the second statement of the lemma, where we shall use Lemma \ref{KutoyantsExtensionLemma}. In the notation of Lemma \ref{KutoyantsExtensionLemma}, let us set $b(t,s,x,\theta)=\sqrt{t}\alpha_{s}^{2}\Phi_{t,s}\zeta(x,\theta)$ and, with some abuse of notation, let us consider the quantity
\begin{align*}
[I_{t}^{2}]&=t\int_{1}^{t}\alpha_{s}^{4}\Phi_{t,s}\zeta(X_{s},\theta_{s})\zeta^{T}(X_{s},\theta_{s})\Phi_{t,s}^{T}ds.
\end{align*}

Now $[I_{t}^{2}]$ is a matrix with elements
\begin{align*}
[I_{t}^{2}]_{i,j}&=t\int_{1}^{t}\alpha_{s}^{4}\sum_{n}\Phi_{t,s,i,n}\sum_{k'=1}^{k}\left(\zeta\zeta^{\top}\right)_{n,k'}(X_{s},\theta_{s})\Phi_{t,s,k',j}ds\nonumber\\
&=t\int_{1}^{\tau_{\delta}\wedge t}\alpha_{s}^{4}\sum_{n}\Phi_{t,s,i,n}\sum_{k'=1}^{k}\left(\zeta\zeta^{\top}\right)_{n,k'}(X_{s},\theta_{s})\Phi_{t,s,k',j}ds
+t\int_{\tau_{\delta}\wedge t}^{\tau_{\delta}\vee t}\alpha_{s}^{4}\sum_{n}\Phi_{t,s,i,n}\sum_{k'=1}^{k}\left(\zeta\zeta^{\top}\right)_{n,k'}(X_{s},\theta_{s})\Phi_{t,s,k',j}ds\nonumber\\
&=[I_{t}^{2}]^{1}_{i,j}+[I_{t}^{2}]^{2}_{i,j}.
\end{align*}

Proceeding as in the first part of the lemma shows that for each index $i,j$, the corresponding element of the matrix $[I_{t}^{2}]$ goes to zero, i.e.,
\begin{align*}
[I_{t}^{2}]_{i,j}&=[I_{t}^{2}]^{1}_{i,j}+[I_{t}^{2}]^{2}_{i,j}\rightarrow 0, \text{ a.s., as  } t\rightarrow \infty.
\end{align*}

To complete the proof, we then need to use Lemma \ref{KutoyantsExtensionLemma}. Let us define
\begin{align*}
D_t &= t^{-c +p-\frac{1}{2}}    \int_{1}^{t } \alpha_s^{2 + p} s^c  \sqrt{t} \Phi_{t,s} \zeta(X_{s},\theta_{s}) E^{\top} ds    \notag \\
&= t^{-c+p}    \int_{1}^{\tau_{\delta}\wedge t }  \alpha_s^{2 + p} s^c   \Phi_{t,s}  \zeta (X_s, \theta_s ) E^{\top} ds  + t^{-c+p}     \int_{\tau_{\delta}\wedge t }^{\tau_{\delta}\vee t }   \alpha_s^{2+p} s^c   \Phi_{t,s}\zeta(X_s, \theta_s ) E^{\top} ds.  \notag \\
&= D_t^1 + D_t^2.
\end{align*}

By Lemma \ref{KutoyantsExtensionLemma}, if we show that for the appropriate choices of $p,c$ and $E$, $D_{t}$ goes to zero in probability as $t\rightarrow\infty$, then we would have shown that the second statement of the lemma holds, i.e. that $I_{t}^{2}\rightarrow 0$ in probability as $t\rightarrow\infty$.

We pick $p=2$, $c=2CC_{\alpha}$ and $E=I$ to be the identity matrix. Let us first start with $D_{t}^{2}$. Notice that for the $(i,j)$ element of the matrix we have
\begin{align}
 \norm{D_t^2}_1  &\leq K t^{-2C C_{\alpha} +2}    \int_{\tau_{\delta}\wedge t }^{\tau_{\delta}\vee t }   \alpha_s^4 s^{2 C C_{\alpha}}   \norm{ \Phi_{t,s} \zeta(X_s, \theta_s )} ds  \notag \\
&\leq  K t^{-3C C_{\alpha} +2}    \int_{\tau_{\delta}\wedge t }^{\tau_{\delta}\vee t }  s^{3 C C_{\alpha}-4}  (1 + \norm{X_{s}}^{m_{1}}+\norm{\theta_{s}}^{m_{2}} ) ds. \notag
\end{align}

It is clear that from this point the analysis of $D_t^2$ is identical to the analysis of $I_t^{1,2}$.  In particular, define the quantity
\begin{align*}
\hat{L}_{t}&=K t^{-3C C_{\alpha} +2}    \int_{1 }^{t }  s^{3 C C_{\alpha}-4}  (1 + \norm{X_{s}}^{m_{1}}+\norm{\theta_{s}}^{m_{2}} ) ds,
\end{align*}
and, for $\varepsilon>0$, consider the event $\hat{A}_{t,\varepsilon}=\left\{ \hat L_{t}\geq t^{\varepsilon-1}\right\}$. Using Markov's inequality and the uniform in time bounds for the moments of $X_{s}$ and $\theta_{s}$ we obtain that

\begin{align*}
\mathbb{P}\left[\hat{A}_{t,\varepsilon}\right]&\leq\frac{\mathbb{E}\left|\hat L_{t}\right|}{t^{\varepsilon-1}}\leq K \frac{t^{-3C C_{\alpha} +2}\int_{1 }^{t }  s^{3 C C_{\alpha}-4}ds  }{t^{\varepsilon-1}}\leq K t^{-\varepsilon}.
\end{align*}

From here on, the rest of the argument follows the Borel-Cantelli argument that was used for the proof of the first part of the lemma. This yields that $D_t^2 \overset{a.s.} \rightarrow 0$ as $t \rightarrow \infty$.

Next, using the semi-group property and for $\tau_{\delta}<t$, $D_t^1$ can be re-written as:
\begin{align}
D_t^1 &=  t^{-2CC_{\alpha} + 2 } \int_{1}^{\tau_{\delta} }  \alpha_s^4 s^{2 CC_{\alpha}}  \Phi_{t,\tau_{\delta}} \Phi_{\tau_{\delta},s}  \zeta (X_s, \theta_s )  ds =t^{-2CC_{\alpha} + 2 } \Phi_{t,\tau_{\delta}} \int_{1}^{\tau_{\delta} }  \alpha_s^4 s^{2 CC_{\alpha}}   \Phi_{\tau_{\delta},s}  \zeta (X_s, \theta_s )  ds.\nonumber
\end{align}

Therefore, by similar logic as for the term $I_{t}^{1,1}$ , we obtain
\begin{eqnarray*}
\norm{ D_t^1} &\leq&  C(\tau_{\delta}) t^{-3 C C_{\alpha} +2 },
\end{eqnarray*}
where $C(\tau_{\delta})$ is almost surely finite.  Therefore, since $CC_{\alpha}>1$, $D_t^1 \overset{a.s.} \rightarrow 0$ as $t \rightarrow \infty$.  Consequently, as $t \rightarrow \infty$, we have indeed obtained
\begin{eqnarray*}
D_t = D_t^1 + D_t^2 \overset{a.s.} \rightarrow 0,
\end{eqnarray*}
which by Lemma \ref{KutoyantsExtensionLemma} implies that the second statement of the lemma is true. This concludes the proof of the lemma.
\end{proof}
\begin{proof}[Proof of Lemma \ref{NonConvexLemmabarSigmaAst}]
$\Phi_{t,s}$ can be expressed in terms of $\Phi_{t,s}^{\ast}$. To see this, first perform a Taylor expansion:
\begin{eqnarray}
d \Phi_{t,s} &=& - \alpha_t \Delta \bar g( \theta^{\ast}) \Phi_{t,s} dt - \alpha_t C_t Y_t  \Phi_{t,s} dt, \qquad
\Phi_{s,s} = I,\notag
\end{eqnarray}
$C_t$ are the third-order partial derivatives of $\bar g (\theta)$ and, based on the imposed assumptions, they are uniformly bounded. Therefore,
\begin{eqnarray}
\Phi_{t,s} = \Phi_{t,s}^{\ast} - \int_s^t \alpha_u \Phi_{t,u}^{\ast} C_u Y_u \Phi_{u,s} du.\notag
\end{eqnarray}

Define $\xi_{t,s} = \int_s^t \alpha_u \Phi_{t,u}^{\ast} C_u Y_u \Phi_{u,s} du$.  Recall that there exists an almost surely finite  random time $\tau_{\delta}$ such that $\norm{ Y_t }  < \delta $ for all $t \geq \tau_{\delta}$ and any $\delta$ small enough.  From the bounds on $\Phi_{t,s}$ and $\Phi_{t,s}^{\ast}$, we have that for $t>s > \tau_{ \delta}$ where $\delta$ small enough,
\begin{eqnarray}
\norm{ \xi_{t,s} }^2 &\leq&  t \int_s^t \alpha_u^2 \norm{ \Phi_{t,u}^{\ast} C_u Y_u \Phi_{u,s} }^2  du \notag \\
&\leq& t  K \int_s^t \alpha_u^2 \norm{ \Phi_{t,u}^{\ast} }^2 \norm{Y_u}^2 \norm{ \Phi_{u,s} }^2 du \notag \\
&\leq& K \delta^2  t^{-2 C C_{\alpha}  +1 } s^{2 C C_{\alpha}}   \int_s^t  u^{-2}  du \notag \\
&\leq& K \delta^2 t^{-2C C_{\alpha} } s^{2C C_{\alpha}}.\notag
\label{XiBound0001}
\end{eqnarray}

We have used the Cauchy-Schwartz inequality in the first inequality above.  Now, consider the case where $t> \tau_{\delta}$ but $s < \tau_{\delta}$.
\begin{eqnarray}
\xi_{t,s} &=& \int_s^t \alpha_u \Phi_{t,u}^{\ast} C_u Y_u \Phi_{u,s} du \notag \\
&=& \int_s^{\tau_{\delta}} \alpha_u \Phi_{t,u}^{\ast} C_u Y_u \Phi_{u,s} du + \int_{\tau_{\delta}}^{t} \alpha_u \Phi_{t,u}^{\ast} C_u Y_u \Phi_{u,s} du \notag \\
&=& \Phi_{t, \tau_{\delta}}^{\ast} \int_s^{\tau_{\delta}} \alpha_u \Phi_{\tau_{\delta}, u}^{\ast} C_u Y_u \Phi_{u,s} du  + \xi_{t, \tau_{\delta}}. \notag
\end{eqnarray}

From the previous bound, $\norm{\xi_{t, \tau_{\delta}}} \leq K \delta t^{-C C_{\alpha} } \tau_{\delta}^{C C_{\alpha}}$.  Therefore, using the fact that $\tau_{\delta}$ is almost surely finite, the bounds on $ \Phi_{t, \tau_{\delta}}^{\ast} $ and $\xi_{t, \tau_{\delta}}$, and the triangle inequality,
\begin{eqnarray}
\norm{\xi_{t,s} } \leq C(\tau_{\delta})  t^{-C C_{\alpha} } ,  \notag
\end{eqnarray}
when  $t> \tau_{\delta}$ but $s < \tau_{\delta}$.

We will use these results to derive the limit of $\bar \Sigma_t$ as $t \rightarrow \infty$.
\begin{eqnarray}
\bar \Sigma_t &=& t \int_1^t \alpha_s^2 \Phi_{t,s} \bar h(\theta^{\ast} )      \Phi_{t,s}^{\top} ds \notag \\
&=& t \int_1^t \alpha_s^2 \Phi_{t,s}^{\ast} \bar h(\theta^{\ast} )      \Phi_{t,s}^{\ast, \top} ds -t \int_1^t \alpha_s^2 \xi_{t,s} \bar h(\theta^{\ast} ) \Phi_{t,s}^{\top} ds -t \int_1^t \alpha_s^2 \Phi_{t,s}^{\ast} \bar h(\theta^{\ast} ) \xi_{t,s}^{\top} ds \notag \\
&=&  \bar \Sigma^{\ast}_t -t \int_1^t \alpha_s^2 \xi_{t,s} \bar h(\theta^{\ast} ) \Phi_{t,s}^{\top} ds -t \int_1^t \alpha_s^2 \Phi_{t,s}^{\ast} \bar h(\theta^{\ast} ) \xi_{t,s}^{\top} ds.
\label{barSigma55}
\end{eqnarray}

Let's consider the second term in (\ref{barSigma55}).  Without loss of generality, assume $t > \tau_{\delta}$.
\begin{eqnarray}
\norm{ t \int_{1}^t \alpha_s^2 \xi_{t,s} \bar h(\theta^{\ast} ) \Phi_{t,s}^{\top} ds}^2 &\leq& t^3 \int_{\tau_{\delta}}^t \norm{ \alpha_s^2 \xi_{t,s} \bar h(\theta^{\ast} ) \Phi_{t,s}^{\top} }^2 ds  +  t^3 \int_{1}^{\tau_{\delta}} \norm{\alpha_s^2 \xi_{t,s} \bar h(\theta^{\ast} ) \Phi_{t,s}^{\top} }^2 ds  \notag \\
&\leq& t^3 \int_{\tau_{\delta}}^t \norm{ \alpha_s^2 \xi_{t,s} \bar h(\theta^{\ast} ) \Phi_{t,s}^{\top} }^2 ds  +  t^3 \int_{1}^{\tau_{\delta}} \norm{\alpha_s^2 \xi_{t,s} \bar h(\theta^{\ast} ) \Phi_{\tau_{\delta},s}^{\top} \Phi_{t, \tau_{\delta}}^{\top} }^2 ds  \notag \\
&\leq& K \delta^2 t^{- 4 C C_{\alpha} + 3} \int_{\tau_{\delta}}^t  s^{4 C C_{\alpha} - 4} ds   + C(\tau_{\delta} ) t^{- 4 C C_{\alpha}  +3 } \notag \\
&\leq& K  \delta^2   + C(\tau_{\delta} ) t^{- 4 C C_{\alpha}  +3 }.
\label{XIbound}
\end{eqnarray}

$C(\tau_{\delta} ) $ is almost surely finite since $\tau_{\delta}$ is almost surely finite. The third term in (\ref{barSigma55}) is similar. Using (\ref{barSigma55}) and the bound (\ref{XIbound}), we have that with probability one:
\begin{eqnarray}
\limsup_{t \rightarrow \infty} \norm{\bar \Sigma_t  - \bar \Sigma^{\ast}_{t} } &\leq& K \delta. \notag
\end{eqnarray}
This implies that $\lim_{t \rightarrow \infty} \norm{\bar \Sigma_t  - \bar \Sigma^{\ast}_{t} }  = 0$  since $\delta$ can be as small as we want. This concludes the proof of the lemma.
\end{proof}
\begin{proof}[Proof of Lemma \ref{NonConvexLemmabarVbarSigma}]
\begin{eqnarray}
 \norm{ \bar V_t - \bar \Sigma_t } &=& \norm{ t \int_1^t \alpha_s^2 \big{(} \Phi_{t,s} \bar h(\theta_s) \Phi_{t,s}^{\top} -  \Phi_{t,s} \bar h(\theta^{\ast}) \Phi_{t,s}^{\top} \big{)} ds}.\notag
  \end{eqnarray}
  Using the semi-group property,
  \begin{eqnarray}
 \bar V_t - \bar \Sigma_t &=&t \int_1^{\tau_{\delta}} \alpha_s^2 \bigg{(}  \Phi_{t,\tau_{\delta} } \Phi_{\tau_{\delta},s} \bar h(\theta_s) \Phi_{\tau_{\delta},s}^{\top} \Phi_{t,\tau_{\delta} }^{\top}  -  \Phi_{t,\tau_{\delta} } \Phi_{\tau_{\delta},s}  \bar h(\theta^{\ast}) \Phi_{\tau_{\delta},s}^{\top} \Phi_{t,\tau_{\delta} }^{\top} \bigg{)} ds \notag \\
 &+& t \int_{ \tau_{\delta} }^{t} \alpha_s^2 \bigg{(} \Phi_{t,s} \bar h(\theta_s) \Phi_{t,s}^{\top} -  \Phi_{t,s} \bar h(\theta^{\ast}) \Phi_{t,s}^{\top} \bigg{)} ds.\notag
  \end{eqnarray}

The $(i,j)$-th element of the matrix $\bar V_t - \bar \Sigma_t$ is:
  \begin{eqnarray*}
 \big{(}  \bar V_t - \bar \Sigma_t  \big{)}_{i,j} &=&  \sum_{m, \ell, n, k' = 1}^k  t \int_{1}^{\tau_{\delta}} \alpha_s^2   \Phi_{t,\tau_{\delta},i,m} \Phi_{\tau_{\delta},s,m, \ell}   \nabla_{\theta}\bar h(\theta_s^1)_{\ell,  n}^{\top} ( \theta_s - \theta^{\ast}) \Phi_{\tau_{\delta},s,n,k'} \Phi_{t, \tau_{\delta},k',j}  \notag \\
 &+& t \int_{\tau_{\delta}}^t \alpha_s^2   \sum_{n, k'} \Phi_{t,s,i,n}   \nabla_{\theta} \bar h(\theta_s^1)_{n,  k'}^{\top} ( \theta_s - \theta^{\ast}) \Phi_{t,s,k',j}  ds.
  \end{eqnarray*}

By the Cauchy-Schwartz inequality and the bound $\norm{\nabla_{\theta} \bar h(\theta)}\leq K(1+\norm{\theta})$ we have
\begin{eqnarray}
&& \bigg{(} t \int_{\tau_{\delta}}^t \alpha_s^2   \sum_{n, k'} \Phi_{t,s,i,n}   \nabla_{\theta} \bar h(\theta_s^1)_{n,  k'}^{\top} ( \theta_s - \theta^{\ast}) \Phi_{t,s,k',j} ds \bigg{)}^2 \notag \\
&\leq& t^3 \int_{\tau_{\delta}}^t \alpha_s^4 \bigg{(}  \sum_{n, k'} \Phi_{t,s,i,n}   \nabla_{\theta} \bar h(\theta_s^1)_{n,  k'}^{\top} ( \theta_s - \theta^{\ast}) \Phi_{t,s,k',j} \bigg{)}^2 ds \notag \\
 &\leq& K \delta^2 t^{- 4 C C_{\alpha} + 3} \int_{\tau_{\delta}}^t  s^{4 C C_{\alpha} - 4} ds   \notag \\
 &\leq& K  \delta^2.
 \label{BoundCS1}
 \end{eqnarray}
We also have that:
 \begin{eqnarray}
 && \bigg{(}   t \int_{1}^{\tau_{\delta}} \alpha_s^2   \sum_{m, \ell, n, k' = 1}^k \Phi_{t,\tau_{\delta},i,m} \Phi_{\tau_{\delta},s,m, \ell} \nabla_{\theta}  \bar h(\theta_s^1)_{\ell,  n}^{\top} ( \theta_s - \theta^{\ast}) \Phi_{\tau_{\delta},s,n,k'}^{\top} \Phi_{t, \tau_{\delta},k',j}^{\top} ds  \bigg{)}^2 \notag \\
 &\leq&  t^3   \int_{1}^{\tau_{\delta}} \alpha_s^4  \bigg{(} \sum_{m, \ell, n, k' = 1}^k    \Phi_{t,\tau_{\delta},i,m} \Phi_{\tau_{\delta},s,m, \ell} \nabla_{\theta}  \bar h(\theta_s^1)_{\ell,  n}^{\top} ( \theta_s - \theta^{\ast}) \Phi_{\tau_{\delta},s,n,k'}  \Phi_{t, \tau_{\delta},k',j}  \bigg{)}^2 ds \notag \\
 &\leq& K t^3 \sum_{m, \ell, n, k' = 1}^k  \bigg{(} \Phi_{t, \tau_{\delta},k',j} \Phi_{t,\tau_{\delta},i,m} \bigg{)}^2     \int_1^{\tau_{\delta}}  \alpha_s^4   \Phi_{\tau_{\delta},s,m, \ell}^2 \bigg{(} \nabla_{\theta} \bar h(\theta_s^1)_{\ell,  n}^{\top} ( \theta_s - \theta^{\ast}) \bigg{)}^2 \Phi_{\tau_{\delta},s,n,k'}^2 ds \notag \\
 &\leq& C(\tau_{\delta} ) t^{-4 C C_{\alpha} + 3}.
 \label{BoundCS2}
 \end{eqnarray}
 Since $\tau_{\delta} < \infty$ with probability $1$, $C(\tau_{\delta})$ is also finite with probability $1$.  Combining the results from equations (\ref{BoundCS1}) and (\ref{BoundCS2}),
 \begin{eqnarray*}
  \big{(}  \bar V_t - \bar \Sigma_t  \big{)}_{i,j} \leq K  \delta^2   + C(\tau_{\delta} ) t^{-4 C C_{\alpha} + 3}.
 \end{eqnarray*}

Therefore, we have that $ \limsup_{t\rightarrow\infty}\norm{ \bar V_t - \bar \Sigma_t } \leq K \delta^{2}$. Since $\delta$ is arbitrarily small, we have that $ \big{(}  \bar V_t - \bar \Sigma_t  \big{)}_{i,j}  \overset{a.s.} \rightarrow 0$, concluding the proof of the lemma.
\end{proof}
\begin{proof}[Proof of Lemma \ref{NonConvexLemmaSigmabarV}]
Notice that we have
\begin{eqnarray}
\Sigma_t - \bar V_t &=& t \int_1^t \alpha_s^2 \bigg{(} \Phi_{t,s} h(\theta_s, X_s)    \Phi_{t,s}^{\top}  - \Phi_{t,s} \bar h(\theta_s)    \Phi_{t,s}^{\top}   \bigg{)} ds.
\label{Poisson0001}
\end{eqnarray}
and consequently the $(i,j)$-th element of the matrix $\Sigma_t - \bar V_t$ is:
\begin{eqnarray*}
 \big{(}  \bar V_t - \bar \Sigma_t  \big{)}_{i,j} &=&  t \int_1^t \alpha_s^2   \sum_{n} \Phi_{t,s,i,n}  \sum_{k'=1}^k \bigg{(} h(\theta_s, X_s)_{n,  k'} -  \bar h(\theta_s)_{n,  k'} \bigg{)}\Phi_{t,s,k',j}^{\top} ds.   \notag
\end{eqnarray*}

Let $w_{n,k'}(x, \theta)$ be the solution to the Poisson equation $\mathcal{L} w_{n,k'}(x,\theta) = h(\theta, x)_{n,  k'} -  \bar h(\theta)_{n,  k'}$.  The solution $w_{n,k'}(t,\theta)$ and its relevant partial derivatives will grow at most polynomially with respect to $\theta$ and $x$ due to the assumptions of Theorem \ref{T:MainTheorem3}.  The next step is to rewrite the difference $\big{(}  \bar V_t - \bar \Sigma_t  \big{)}_{i,j}$ using that Poisson equation and It\^{o}'s formula and then to show that each term on the right hand side of the resulting equation goes to zero.

For example we shall have that
\begin{eqnarray}
&t& \int_1^t \alpha_s^2   \sum_{n} \Phi_{t,s,i,n}  \sum_{k'=1}^k \bigg{(} h(\theta_s, X_s)_{n,  k'} -  \bar h(\theta_s)_{n,  k'} \bigg{)} \Phi_{t,s,k',j} ds \notag \\
&=&  t \int_1^t \alpha_s^2   \sum_{n} \Phi_{t,s,i,n}  \sum_{k'=1}^k  \Phi_{t,s,k',j}  \nabla_x w_{n, k'}(X_s, \theta_s )^{\top} d W_s + (\ast),
\end{eqnarray}
where $(\ast)$ is a collection of Riemann integrals resulting from the application of It\^{o}'s formula. Now each of these terms can be shown to go to zero with an argument exactly parallel to Lemma \ref{L:MainTermsConvergenceZero}. Due to the similarity of the argument, the details are omitted.

\end{proof}

\section{Convergence Analysis and Insights} \label{ConvergenceAnalysis}
The central limit theorem provides an important theoretical guarantee for the performance of the SGDCT algorithm developed in \cite{SGDCT1}.  Theorem  \ref{T:MainTheorem3} is particularly significant since it shows that the asymptotic convergence rate of $t^{- \frac{1}{2}}$ even holds for a certain class of non-convex models.  This is important since many models are non-convex.

In addition, the analysis yields insight into the behavior of the algorithm and provides guidance on selecting the optimal learning rate for numerical performance.  The regime where the central limit theorem holds with the optimal rate $\sqrt{t}$ is $C_{\alpha} C > 1$.  $C_{\alpha}$ is the magnitude of the learning rate.  For example, take $\alpha_t = \frac{C_{\alpha}}{C_0 + t}$.  Therefore, the learning rate magnitude must be chosen sufficiently large in order to achieve the optimal rate of convergence.  The larger the constant $C$ is, the steeper the function $\bar g (\theta)$ is around the global minimum $\theta^{\ast}$.  The smaller the constant $C$ is, the smaller the function $\bar g (\theta)$ is around the global minimum $\theta^{\ast}$.  The ``flatter" the region around the global minimum point, the larger the learning rate magnitude must be.  If the region around the global minimum point is steep, the learning rate magnitude can be smaller.


The condition of $C_{\alpha} C > 1$ to ensure the convergence rate of $t^{- \frac{1}{2}}$ is not specific to the SGDCT algorithm, but is in general a characteristic of continuous-time statistical learning algorithms.  The convergence rate of any continuous-time gradient descent algorithm with a decaying learning rate will depend upon the learning rate magnitude $C_{\alpha}$.  Consider the deterministic gradient descent algorithm
\begin{eqnarray*}
\frac{d \theta_t}{d t} = - \alpha_t \nabla_{\theta} \textcolor{black}{ \bar g (\theta_t)}.
\end{eqnarray*}
Let $\alpha_t = \frac{C_{\alpha}}{C_0 + t}$ and assume $\bar g (\theta)$ is strongly convex.  Then $\norm{ \theta_t - \theta^{\ast} } \leq K t^{- C C_{\alpha}  }$.
Note that the convergence rate depends entirely upon the choice of the learning rate magnitude $C_{\alpha}$.  If $C_{\alpha}$ is very small, the deterministic gradient descent algorithm will even converge at a rate much smaller than $t^{-\frac{1}{2}}$.

$t^{-\frac{1}{2}}$ is the fastest possible convergence rate given that the noise in the system (\ref{ClassofEqns}) is a Brownian motion.  This is due to the variance of a Brownian motion growing linearly in time.  However, other types of noise with variances which grow sub-linearly in time could allow for a faster rate of convergence than $t^{- \frac{1}{2}}$.  An example of a stochastic process whose variance grows sub-linearly in time is a fractional Brownian motion with appropriately chosen Hurst parameter.  Analyzing the convergence rate under more general types of noise would be a very interesting topic for future research.

In the central limit theorem result, we are also able to precisely characterize the asymptotic variance-covariance matrix $\bar \Sigma=\left[\bar\Sigma_{i,j}\right]_{i,j=1}^{k}$, with
\begin{eqnarray*}
\bar \Sigma_{i,j} =  \sum_{n}\sum_{m} u_{i,m}u_{n,m}\sum_{k'} \bar h(\theta^{\ast})_{n,  k'} \sum_{m'} \left[ \frac{C_{\alpha}^{2}}{\left(\lambda_{m}+\lambda_{m'}\right)C_{\alpha}-1}\right] u_{j,m'}u_{k',m'}
\end{eqnarray*}
and
\begin{equation*}
\bar{\Sigma}=C_{\alpha}^{2}\int_{0}^{\infty}e^{-s\left(C_{\alpha}\Delta\bar{g}(\theta^{\ast})-I\right)}\bar{h}(\theta^{\ast})e^{-s\left(C_{\alpha}\left(\Delta\bar{g}\right)^{\top}(\theta^{\ast})-I\right)}ds.
\end{equation*}

The covariance depends upon the eigenvalues and eigenvectors of the matrix $\Delta \bar g (\theta^{\ast} )$, which is the Hessian matrix $\Delta \bar g (\theta )$ at the global minimum $\theta^{\ast}$.  The larger the eigenvalues, the smaller the variance.  This means that the steeper the function $\bar g(\theta)$ is near the global minimum $\theta^{\ast}$, the smaller the asymptotic variance.  The flatter the function $\bar g (\theta)$ is near the global minimum, the larger the asymptotic variance.  If the function is very flat, $\theta_t$'s drift towards $\theta^{\ast}$ is dominated by the fluctuations from the noise $W_t$.  The covariance also depends upon the learning rate magnitude $C_{\alpha}$.  The larger the learning rate magnitude, the larger the asymptotic variance $\bar \Sigma$.  Although a sufficiently large learning rate is required to achieve the optimal rate of convergence $t^{-\frac{1}{2}}$, too large of a learning rate will cause high variance.

%


\section{Conclusion} \label{Conclusion}
Stochastic gradient descent in continuous time (SGDCT) provides a computationally efficient method for the statistical learning of continuous-time models, which are widely used in science, engineering, and finance.  The SGDCT algorithm follows a (noisy) descent direction along a continuous stream of data.  The algorithm updates satisfy a stochastic differential equation.  This paper analyzes the asymptotic convergence rate of the SGDCT algorithm by proving a central limit theorem.  An L$^p$ convergence rate is also proven for the algorithm.  

In addition to a theoretical guarantee, the convergence rate analysis provides important insights into the behavior and dynamics of the algorithm.  The asymptotic covariance is precisely characterized and shows the effects of different features such as the learning rate, the level of noise, and the shape of the objective function.


The proofs in this paper require addressing several challenges.  First, fluctuations of the form $\int_{0}^{t}\alpha_s \big{(} h(X_s, \theta_s ) - \bar h(\theta_s ) \big{)}ds$ must be analyzed.  We evaluate, and control with rate $\alpha_t^2$, these fluctuations using a Poisson partial differential equation.  Secondly, the model $f(x, \theta)$ is allowed to grow with $\theta$.  This means that the fluctuations as well as other terms can grow with $\theta$.  Therefore, we must prove an a priori stability estimate for $\norm{ \theta_t } $.  Proving a central limit theorem for the non-convex $\bar g(\theta)$ in Theorem \ref{T:MainTheorem3} is challenging since the convergence speed of $\theta_t$ can become arbitrarily slow in certain regions, and the gradient can even point away from the global minimum $\theta^{\ast}$.  We prove the central limit theorem for the non-convex case by analyzing two regimes $[0, \tau_{\delta} ]$ and $[\tau_{\delta}, \infty)$, where $\tau_{\delta}$ is defined such that $\norm{ \theta_t - \theta^{\ast} } < \delta $ for all $t \geq \tau_{\delta}$.  


\appendix
\section{Preliminary Estimates}  \label{Preliminary}
This section presents two key bounds that are used throughout the  paper.  Section \ref{StabilityBound} proves uniform in time moment bounds for  $\theta_t$.  That is, we prove that $\mathbb{E}[ \norm{ \theta_t }^p ]$ is bounded uniformly in time.  Section \ref{PoissonPDEbound} presents a bound on the solutions for a class of Poisson partial differential equations. In the paper, we relate certain equations to the solution of a Poisson partial differential equation and then apply this bound.

\subsection{Moment bounds} \label{StabilityBound}

It is easy to see that for $\norm{\theta}>R$ (in fact for $\norm{\theta}>0$) one has for $\norm{\theta_{t}}=\sqrt{\sum_{i=1}^{k}\left|\theta_{t}^{i}\right|^{2}}$
\begin{align*}
d\norm{\theta_{t}}&=\left[-\alpha_{t}\frac{\left<\theta_{t},\nabla_{\theta}g(X_{t},\theta_{t})\right>}{\norm{\theta_{t}}}-
\alpha_{t}^{2}\sum_{i,j=1}^{n}\frac{\theta_{t}^{i}\theta_{t}^{j}\left(\nabla_{\theta}f\nabla^{T}_{\theta}f(X_{t},\theta_{t})\right)_{i,j}}{2\norm{\theta_{t}}^{3}}
+\alpha_{t}^{2}\sum_{i=1}^{n}\frac{\left(\nabla_{\theta}f\nabla^{T}_{\theta}f(X_{t},\theta_{t})\right)_{i,i}}{2\norm{\theta_{t}}}\right]dt\nonumber\\
&\quad+\alpha_{t}\frac{1}{\norm{\theta_{t}}}\left<\theta_{t},\nabla_{\theta}f(X_{t},\theta_{t})dW_{t}\right>\nonumber\\
&=\left[-\alpha_{t}\frac{\left<\theta_{t},\nabla_{\theta}g(X_{t},\theta_{t})\right>}{\norm{\theta_{t}}}-
\alpha_{t}^{2}\sum_{i,j=1}^{n}\frac{\theta_{t}^{i}\theta_{t}^{j}\left(\nabla_{\theta}f\nabla^{T}_{\theta}f(X_{t},\theta_{t})\right)_{i,j}}{2\norm{\theta_{t}}^{3}}
+\alpha_{t}^{2}\sum_{i=1}^{n}\frac{\left(\nabla_{\theta}f\nabla^{T}_{\theta}f(X_{t},\theta_{t})\right)_{i,i}}{2\norm{\theta_{t}}}\right]dt\nonumber\\
&\quad+\alpha_{t}\tau (X_{t},\theta_{t})d\tilde{W}_{t},
\end{align*}
for an independent one-dimensional Brownian motion $\tilde{W}$. Let us also now consider the process $\tilde{\theta}_{t}$ which satisfies
\begin{align*}
d\tilde{\theta}_{t}&=-\alpha_{t}\kappa(X_{t})\tilde{\theta}_{t}dt+\alpha_{t}\nabla_{\theta}f(X_{t},\tilde{\theta}_{t})dW_{t},
\end{align*}
where $\kappa(x)$ is from Condition \ref{A:RecurrenceCondition0}. Then, we get similarly that for $\norm{\tilde{\theta}}>0$ and $\norm{\tilde\theta_{t}}=\sqrt{\sum_{i=1}^{k}\left|\tilde\theta_{t}^{i}\right|^{2}}$
\begin{align*}
d\norm{\tilde{\theta}_{t}}&=\left[-\alpha_{t}\kappa(X_{t})\norm{\tilde{\theta}_{t}}-
\alpha_{t}^{2}\sum_{i,j=1}^{n}\frac{\tilde{\theta}_{t}^{i}\tilde{\theta}_{t}^{j}\left(\nabla_{\tilde{\theta}}f\nabla^{T}_{\tilde{\theta}}f(X_{t},\tilde{\theta}_{t})\right)_{i,j}}{2\norm{\tilde{\theta}_{t}}^{3}}
+\alpha_{t}^{2}\sum_{i=1}^{n}\frac{\left(\nabla_{\tilde{\theta}}f\nabla^{T}_{\tilde{\theta}}f(X_{t},\tilde{\theta}_{t})\right)_{i,i}}{2\norm{\tilde{\theta}_{t}}}\right]dt\nonumber\\
&\quad+\alpha_{t}\frac{1}{\norm{\tilde{\theta}_{t}}}\left<\tilde{\theta}_{t},\nabla_{\tilde{\theta}}f(X_{t},\tilde{\theta}_{t})dW_{t}\right>\nonumber\\
&=\left[-\alpha_{t}\kappa(X_{t})\norm{\tilde{\theta}_{t}}-
\alpha_{t}^{2}\sum_{i,j=1}^{n}\frac{\tilde{\theta}_{t}^{i}\tilde{\theta}_{t}^{j}\left(\nabla_{\tilde{\theta}}f\nabla^{T}_{\tilde{\theta}}f(X_{t},\tilde{\theta}_{t})\right)_{i,j}}{2\norm{\tilde{\theta}_{t}}^{3}}
+\alpha_{t}^{2}\sum_{i=1}^{n}\frac{\left(\nabla_{\tilde{\theta}}f\nabla^{T}_{\tilde{\theta}}f(X_{t},\tilde{\theta}_{t})\right)_{i,i}}{2\norm{\tilde{\theta}_{t}}}\right]dt\nonumber\\
&\quad+\alpha_{t}\tau(X_{t},\tilde{\theta}_{t})d\tilde{W}_{t}
\end{align*}

Due to Conditions  \ref{A:RecurrenceCondition0} and \ref{A:GrowthConditions0} and continuity of the involved drift and diffusion coefficients for $\norm{\theta},\norm{\tilde{\theta}}> R >0 $, we may use the comparison theorem (see for example \cite{IkedaWatanabe1977}) to obtain that
\begin{align}
\mathbb{P}\left(\norm{\theta_{t}}\leq \norm{\tilde{\theta}_{t}}, t\geq 0\right)=1.\label{Eq:AlmostSureDomination}
\end{align}

It is easy to see that the proof of the comparison theorem 1.1 of \cite{IkedaWatanabe1977} goes through almost verbatim, despite the presence of the term $\lambda(x)$ in Condition \ref{A:GrowthConditions0}. The reason is that $|\lambda(x)|$ is assumed to have at most polynomial growth in $\norm{x}$ and all moments of $X_{t}$ are bounded uniformly in $t$.

Now, notice  that $\tilde{\theta}_{t}$ can be written as the solution to the integral equation
\begin{align*}
\tilde{\theta}_{t}&=\tilde{\theta}_{0}e^{-\int_{0}^{t}\alpha_{s}\kappa(X_{s})ds}+\int_{0}^{t}\alpha_{s}e^{-\int_{s}^{t}\alpha_{r}\kappa(X_{r})dr}\nabla_{\theta}f(X_{s},\tilde{\theta}_{s})dW_{s}.
\end{align*}

From the latter representation we obtain, recall that $\kappa(x)$ is almost surely positive, that for any $p\geq 1$
\begin{align*}
\mathbb{E}\norm{\tilde{\theta}_{t}}^{2p}&\leq K \mathbb{E}\norm{\tilde{\theta}_{0}}^{2p}+ K \mathbb{E} \norm{\int_{0}^{t}\alpha_{s}e^{-\int_{s}^{t}\alpha_{r}\kappa(X_{r})dr}\nabla_{\theta}f(X_{s},\tilde{\theta}_{s})dW_{s} }^{2p}\nonumber\\
&\leq K \mathbb{E}\norm{\tilde{\theta}_{0}}^{2p}+ K \mathbb{E}\left(\int_{0}^{t}\alpha^{2}_{s}e^{-2\int_{s}^{t}\alpha_{r}\kappa(X_{r})dr}\norm{\nabla_{\theta}f(X_{s},\tilde{\theta}_{s})}^{2}ds\right)^{p}\nonumber\\
&\leq K \mathbb{E}\norm{\tilde{\theta}_{0}}^{2p}+K \mathbb{E}\int_{0}^{t}\alpha^{2p}_{s}e^{-2p\int_{s}^{t}\alpha_{r}\kappa(X_{r})dr}\norm{\nabla_{\theta}f(X_{s},\tilde{\theta}_{s})}^{2p}ds\nonumber\\
&\leq K \mathbb{E}\norm{\tilde{\theta}_{0}}^{2p}+ K \mathbb{E}\int_{0}^{t}\alpha^{2p}_{s}\norm{\nabla_{\theta}f(X_{s},\tilde{\theta}_{s})}^{2p}ds\nonumber\\
&\leq K \mathbb{E}\norm{\tilde{\theta}_{0}}^{2p}+K \int_{0}^{t}\alpha^{2p}_{s}\left(1+\mathbb{E}\norm{X_{s}}^{2pq}+\mathbb{E}\norm{\tilde{\theta}_{s}}^{2p}\right)ds\nonumber\\
&\leq K \mathbb{E}\norm{\tilde{\theta}_{0}}^{2p}+K \int_{0}^{t}\alpha^{2p}_{s}ds+ K \int_{0}^{t}\alpha^{2p}_{s}\mathbb{E}\norm{\tilde{\theta}_{s}}^{2p}ds\nonumber\\
&\leq K \mathbb{E}\norm{\tilde{\theta}_{0}}^{2p}+K\frac{1}{2p-1}\left(C_{0}^{1-2p}-(C_{0}+t)^{1-2p}\right) +K \int_{0}^{t}\alpha^{2p}_{s}\mathbb{E}\norm{\tilde\theta_{s}}^{2p}ds.\nonumber\\
&\leq K+ K \int_{0}^{t}\alpha^{2p}_{s}\mathbb{E}\norm{\tilde\theta_{s}}^{2p}ds,
\end{align*}
where the unimportant finite constant $K<\infty$ changes from line to line.  Hence, Gronwall lemma then immediately gives that for any $p\geq 1$ there exists a finite constant $K<\infty$ such that
\begin{align}
\sup_{t>0}\mathbb{E}\norm{\tilde{\theta}_{t}}^{2p}&\leq K.\label{Eq:OU_BoundedMoments}
\end{align}

Combining (\ref{Eq:AlmostSureDomination}) and (\ref{Eq:OU_BoundedMoments}) we then obtain that for any $p\geq1$ and for an appropriate finite constant $K$,
\begin{align*}
\sup_{t>0}\mathbb{E}\norm{\theta_{t}}^{2p}&\leq K,
\end{align*}
completing the proof of the targeted bound.

\subsection{Poisson PDE} \label{PoissonPDEbound}
We recall the following regularity result from \cite{pardoux2003poisson} on the Poisson equations in the whole space, appropriately stated to cover our case of interest.
\begin{theorem} \label{T:RegularityPoisson}
Let Conditions \ref{A:LyapunovCondition}, \ref{A:Assumption0} and \ref{A:Assumption1} be satisfied. Assume that $H(x,\theta)\in \mathcal{C}^{\alpha,2}\left(\mathcal{X},\mathbb{R}^{k}\right)$ \textcolor{black}{where $\mathcal{X} \subseteq \mathbb{R}^m$},
\begin{equation}
\int_{\mathcal{X}} H(x,\theta)\pi(dx)=0,\label{Eq:CenteringCondition}
\end{equation}
and that for some positive constants $K$, $p_{1},p_{2},p_{3}$ and $q$,
\begin{eqnarray}
\norm{ H(x, \theta) } &\leq& K (1 + \norm{\theta}^{p_{1}} ) ( 1 + \norm{x}^q ), \notag \\
 \norm{ \frac{\partial H}{\partial \theta} (x, \theta)  } &\leq& K (1 + \norm{\theta}^{p_{2}} ) (1+\norm{x}^{q} ), \notag \\
 \norm{ \frac{\partial^2 H}{\partial \theta^2} (x, \theta)  } &\leq& K (1 + \norm{\theta}^{p_{3}} ) (1+\norm{x}^{q} ).
\label{Hbounds0}
\end{eqnarray}
Let $\mathcal{L}_{x}$ be the infinitesimal generator for the $X$ process. Then the Poisson equation
\begin{eqnarray}
& &\mathcal{L}_{x}u(x,\theta)= H(x,\theta),\quad\int_{\mathcal{X}}%
u(x,\theta)\pi(dx)=0 \label{Eq:CellProblem}
\end{eqnarray}
has a unique solution that satisfies $u(x,\cdot)\in \mathcal{C}^{2}$ for every $x\in\mathcal{X}$, $\partial_{\theta}^{2}u\in \mathcal{C}\left(\mathcal{X}\times\mathbb{R}^{n}\right)$ and there exist positive constants $K'$ and $m$ such that
\begin{eqnarray}
&&  \norm{u (x, \theta)} + \norm{ \nabla_x u (x, \theta) }  \leq  K' ( 1+  \norm{ \theta }^{p_{1}} ) (1 + \norm{x}^m) ,  \notag \\
&& \norm{ \frac{\partial u}{\partial \theta }(x, \theta) } + \norm{ \frac{\partial^2 u}{\partial x \partial \theta}(x, \theta) }     \leq  K' (1 + \norm{\theta}^{p_{2}}) ( 1 + \norm{x}^m ), \notag \\
&&  \norm{ \frac{\partial^2 u}{\partial \theta^2 }(x, \theta) }  + \norm{ \frac{\partial^3 u}{\partial x \partial \theta^2}(x, \theta) }  \leq  K' (1 + \norm{\theta}^{p_{3}} ) ( 1 + \norm{x}^m ).
\label{RegularityBounds1}
\end{eqnarray}
\end{theorem}

Terms of the form $H(x, \theta) = h(x, \theta) - \bar h( \theta)$ must be analyzed and controlled throughout the paper.  Note that $H(x, \theta)$ satisfies the centering condition (\ref{Eq:CenteringCondition}).  For example, the term $G(x, \theta) = \nabla_{\theta} g (x, \theta) - \nabla_{\theta} \bar g ( \theta)$ needs to be controlled. $G(x, \theta)$ satisfies the centering condition (\ref{Eq:CenteringCondition}) and Condition \ref{A:Assumption1} implies the bounds (\ref{Hbounds0}) with appropriate choices for $p_{1},p_{2},p_{3}$ and $q$.  Therefore, the Poisson solution (\ref{Eq:CellProblem}) associated with $H(x, \theta) = G(x, \theta)$ satisfies the bounds (\ref{RegularityBounds1}).

\section{Proof of Convergence with Linear Growth in $f(x, \theta)$ }  \label{ConvergenceProofQuadraticGrowth}
Under the global Lipschitz assumption on $\nabla\bar{g}(\theta)$ and the uniform bounds on the moments of $\theta$ derived in Appendix \ref{StabilityBound} the proof is exactly the same as in \cite{SGDCT1}, except for the term $J_t^{(1)}$ of Lemma 3.1, in Section 3 of \cite{SGDCT1} which we re-define below and prove converges almost surely to $0$.
\[
J_{t}^{(1)}=\alpha_{t}    \left\|v(X_{t},\theta_{t})\right\|.
\]

Using the bounds (\ref{RegularityBounds1}) and results from \cite{PardouxVeretennikov1}, there is some $0<K<\infty$ (that may change from line to line below) and $0<q<\infty$ such that for $t$ large enough
\begin{align*}
\mathbb{E}|J_{t}^{(1)}|^{2}&\leq K \alpha_{t}^{2}\mathbb{E}\left[1+ \|X_{t}\|^{q} + \| \theta_t \|^2 \right] \leq K \alpha_t^2.
\end{align*}

Consider $p>0$ such that $\lim_{t\rightarrow\infty}\alpha_{t}^{2}t^{2p}=0$ and for any $\delta\in(0,p)$ define the event $A_{t,\delta}=\left\{J_{t}^{(1)}\geq t^{\delta-p}\right\}$. Then we have for $t$ large enough such that $\alpha_{t}^{2}t^{2p}\leq 1$
\[
\mathbb{P}\left(A_{t,\delta}\right)\leq \frac{\mathbb{E}|J_{t}^{(1)}|^{2}}{t^{2(\delta-p)}}\leq K\frac{\alpha_{t}^{2}t^{2p}}{t^{2\delta}}\leq K\frac{1}{t^{2\delta}}.
\]

The latter implies that
\[
\sum_{n\in\mathbb{N}}\mathbb{P}\left(A_{2^{n},\delta}\right)<\infty.
\]

Therefore, by Borel-Cantelli lemma we have that for every $\delta\in(0,p)$ there is a finite positive random variable $d(\omega)$ and some $n_{0}<\infty$ such that for every $n\geq n_{0}$ one has
\[
J_{2^{n}}^{(1)}\leq\frac{d(\omega)}{2^{n(p-\delta)}}.
\]

Thus for $t\in[2^{n},2^{n+1})$ and $n\geq n_{0}$ one has for some finite constant $K<\infty$
\[
J_{t}^{(1)} \leq K \alpha_{2^{n+1}} \sup_{s\in(0,2^{n+1}]}\|v(X_{s},\theta_{s})\| \leq K \frac{d(\omega)}{2^{(n+1)(p-\delta)}}\leq K \frac{d(\omega)}{t^{p-\delta}}.
\]

The latter display then guarantees that for $t\geq 2^{n_{0}}$ we have with probability one
\begin{align*}
J_{t}^{(1)}\leq K \frac{d(\omega)}{t^{p-\delta}}\rightarrow 0, \text{ as }t\rightarrow\infty.
\end{align*}

\end{document}